\pdfoutput=1

\documentclass[hidelinks,onefignum,onetabnum,a4paper]{siamart220329}

\usepackage{tikz}
\usepackage{graphicx,epstopdf}
\usepackage[caption=false]{subfig}

\usepackage{lipsum}
\usepackage{amsfonts}
\usepackage{graphicx}
\usepackage{epstopdf}
\usepackage{algorithmic}
\ifpdf
  \DeclareGraphicsExtensions{.eps,.pdf,.png,.jpg}
\else
  \DeclareGraphicsExtensions{.eps}
\fi

\newsiamremark{remark}{Remark}
\newsiamremark{hypothesis}{Hypothesis}
\crefname{hypothesis}{Hypothesis}{Hypotheses}
\newsiamthm{claim}{Claim}

\headers{Space-time parallel iterative solvers for parabolic problems}{A. Arrar\'{a}s, F. J. Gaspar, I. Jimenez-Ciga, L. Portero}

\title{Space-time parallel iterative solvers for the integration of parabolic problems\thanks{Submitted to the editors June 30, 2024.
\funding{The work of I\~nigo Jimenez-Ciga was supported by Public University of Navarre (PhD Grant). The work of all the authors was supported by Grant PID2022-140108NB-I00 funded by MCIN/AEI/10.13039/501100011033 and by ``ERDF A way of making Europe''.}}}

\author{Andr\'{e}s Arrar\'{a}s\thanks{Department of Statistics, Computer Science and Mathematics, Public University of Navarre, Pamplona, Spain 
  (\email{andres.arraras@unavarra.es}, \email{inigo.jimenez@unavarra.es}, \email{laura.portero@unavarra.es}).}
\and Francisco J. Gaspar\thanks{Department of Applied Mathematics, University of Zaragoza, Zaragoza, Spain
  (\email{fjgaspar@unizar.es}).}
\and I\~{n}igo Jimenez-Ciga\footnotemark[2]
\and Laura Portero\footnotemark[2]}

\usepackage{amsopn}

\ifpdf
\hypersetup{
	pdftitle={Space-time parallel iterative solvers for the integration of parabolic problems},
	pdfauthor={A. Arrar\'{a}s, F. J. Gaspar, I. Jimenez-Ciga, L. Portero}
}
\fi

\begin{document}
	
	\maketitle
	
	\begin{abstract}
		In view of the existing limitations of sequential computing, parallelization has emerged as an alternative in order to improve the speedup of numerical simulations. In the framework of evolutionary problems, space-time parallel methods offer the possibility to optimize parallelization. In the present paper, we propose a new family of these methods, built as a combination of the well-known parareal algorithm and suitable splitting techniques which permit us to parallelize in space. In particular, dimensional and domain decomposition splittings are considered for partitioning the elliptic operator, and first-order splitting time integrators are chosen as the propagators of the parareal algorithm to solve the resulting split problem.
		
		The major contribution of these methods is that, not only does the fine propagator perform in parallel, but also the coarse propagator. Unlike the classical version of the parareal algorithm, where all processors remain idle during the coarse propagator computations, the newly proposed schemes utilize the computational cores for both integrators. A convergence analysis of the methods is provided, and several numerical experiments are performed to test the solvers under consideration.
	\end{abstract}
	
	
	
	\section{Introduction}
	In recent years, numerical simulations have become a useful tool in order to analyze mathematical models which rule phenomena that belong to many fields of knowledge. These simulations are computed by increasingly complex computer architectures, aiming to meet the demand for faster numerical results. Nowadays, the solution of most problems of practical interest are burdensome, on account of the number of operations and memory involved in the process. Given the existing limits of improvement of clock speeds, parallelization has emerged as a means to increase the speedup of simulations.
	
	In the context of evolutionary problems, initial attempts at parallelization focused on the spatial variables. However, given the increasing number of cores in modern supercomputers, parallelization in time has been considered since the 1960s (cf. \cite{Nievergelt1964}) as a way to optimize the use of available processors. For a general survey on the most relevant parallel-in-time integrators, as well as a detailed classification of them, we refer the reader to \cite{Gander2015} and \cite{OngSchroder2020}.
	
	This work considers the parareal algorithm as the time parallel time integrator of the proposed methods. First introduced in \cite{LionsMadayTurinici2001}, it is based on two so-called propagators, namely, the coarse and fine propagators. While the former is cheap and inaccurate, the latter is expensive but accurate. The main advantage of this algorithm is that this last propagator can be computed in parallel. For a deep analysis of the stability and convergence properties of the algorithm, we refer the reader to the seminal works \cite{GanderVandewalle2007} and \cite{StaffRonquist2005}.
	
	On the other hand, space parallelization is implemented using splitting techniques. The key idea is to partition the elliptic operator into simpler suboperators so that, by using time-splitting methods, the resulting linear systems become uncoupled. This allows the ensuing subsystems to be solved independently. We consider two splitting techniques for the partition of the elliptic operator, namely, dimensional and domain decomposition splittings. The former provides a partition based on the spatial variables that involve the elliptic term (cf. \cite{HundsdorferVerwer2003}), whereas the latter considers a suitable domain decomposition, and a subsequent partition of unity subordinate to it, to define the suboperators that form the splitting (see \cite{MathewPolyakovRussoWang1998} for more details). References \cite{ArrarasintHoutHundsdorferPortero2017} and \cite{ArrarasPortero2015} show some applications of domain decomposition splittings. In order to solve the resulting partitioned problems, this paper considers two first-order splitting time integrators, the fractional implicit Euler scheme (cf. \cite{Marchuk1990}) and the Douglas-Rachford method (cf. \cite{DouglasRachford1956}).
	
	The main contribution of our proposal is to consider those time-splitting integrators as the coarse and fine propagators of the parareal algorithm. Although a first approach was proposed in \cite{ArrarasGasparPorteroRodrigo2023} for the fractional implicit Euler scheme as both propagators of the parareal algorithm, in this paper we extend this idea in order to consider also the Douglas-Rachford method. Moreover, a rigorous theoretical analysis of the convergence properties is provided, together with a rich variety of numerical test experiments. There are two main reasons that justify the use of splitting techniques combined with the parareal algorithm. First, the fine propagator can be computed faster. In addition, the coarse propagator becomes even cheaper, and all the processors, which in the classical version remain unused for the computation of the coarse propagator, can be devoted to implementing space parallelization of such propagator. We prove the convergence of the newly proposed methods using the concept of convergence factor, first introduced by Gander and Vandewalle in \cite{GanderVandewalle2007}, yielding similar results to the ones obtained in \cite{Wu2015}, \cite{Wu2016}, and \cite{WuZhou2015}.
	
	The parareal algorithm is known for being a simple parallel-in-time method, while efficient for the integration of parabolic problems. However, it has a significant drawback, that is, it is not suitable for hyperbolic problems. In this paper, we consider a general evolutionary reaction-diffusion problem with linear reaction term. After following the method of lines, the proposed methods are implemented in order to solve the resulting semidiscrete problem. In addition, the robustness and scalability of the solvers are analyzed, performing certain numerical test experiments.
	
	The rest of the paper is organized as follows. In \Cref{sec:split}, we introduce the model problem and the splitting techniques that will permit us to parallelize in space. \Cref{sec:par} provides the formulation and theoretical details of the parareal algorithm, and the convergence factor is defined for the partitioned Dahquist test equation. In \Cref{sec:conv}, we formulate the space-time parallel methods and their convergence properties are proven. Finally, \Cref{sec:num} contains some illustrative numerical experiments and some concluding remarks are shown in \Cref{sec:concl}.
	
	\section{Model problem and splitting techniques}
	\label{sec:split}
	
	In this section, we propose the particular model problem under consideration in the present paper. Moreover, we describe two partitioning strategies for its elliptic term. Then, once the problem is discretized in space following the method of lines, we consider two first-order splitting time schemes in order to integrate the resulting partitioned semidiscrete problem. The linear systems that arise at each internal stage of the discrete problem are in fact uncoupled systems that can be solved in parallel.
	
	\subsection{Model problem and partitioning techniques}
	\label{susec:tyspl}
	
	Let us consider the following initial-boundary value problem
	\begin{equation} \label{cont_problem}
		\left\{ \begin{array}{l l}
			u_t = Lu + f, & \text{in } \Omega \times (0,T], \\[0.5ex]
			u = g, & \text{on } \Gamma \times (0,T], \\[0.5ex]
			u = u_0, & \text{in } \overline{\Omega} \times \{0\},
		\end{array} \right.
	\end{equation}
	where $\Omega \subset \mathbb{R}^d$ is a bounded connected domain, whose boundary is denoted by $\Gamma = \partial \Omega$. Moreover, let us assume sufficient smoothness on $u = u(\mathbf{x},t)$, $f = f(\mathbf{x},t)$, $g = g(\mathbf{x},t)$, and $u_0 = u_0(\mathbf{x})$, together with suitable compatibility conditions. Finally, let us consider $Lu = \nabla \cdot (D \nabla u) - cu$ as the elliptic operator of the problem, with $c > 0$ and $D = D(\mathbf{x})$ a $d \times d$ symmetric positive definite tensor, which, for some $0 < \kappa_1 \leq \kappa_2 < \infty$, satisfies
	\begin{displaymath}
		\kappa_1\xi^T\xi \leq \xi^T D(\mathbf{x})\xi \leq \kappa_2 \xi^T \xi,
	\end{displaymath} 
	for every $\mathbf{x} \in \Omega$ and $\xi \in \mathbb{R}^d$.
	
	We propose suitable partitions of the elliptic operator, i.e., $L = L_1 + \cdots + L_M$, such that the resulting split problems can be solved applying parallelization. 
	In particular, we consider two splitting strategies, namely, dimensional and domain decomposition splittings.
	
	First, dimensional splitting is a rather classical splitting technique which provides a partition of the elliptic operator into $M = d$ suboperators. Basically, each suboperator comprises the partial derivatives which correspond to each of the spatial variables (cf. \cite{HundsdorferVerwer2003}). A notorious drawback of this partitioning strategy is that mixed partial derivatives cannot be handled directly (for classical schemes that cope with this issue, we refer the reader to \cite{MckeeMitchell1970}, \cite{Samarskii1964}, or \cite{Sofronov1963}). Furthermore, the implementation of dimensional splitting becomes a handicap for complex spatial discretizations or domains. Taking these restrictions into account, let us assume that tensor $D$ is diagonal, i.e., $D = \text{diag}(d_{11},d_{22},\dots,d_{MM})$. Then, the          dimensional splitting for operator $L$ can be defined as follows
	\begin{displaymath}
		L_ju = (d_{jj} u_{x_j})_{x_j} - \frac{1}{M}cu, \quad j \in \{1,\dots,M\},
	\end{displaymath}
	where $x_j$ denotes the $j$-th spatial variable. Notice that the linear reactive term is split into $M$ equal terms added to the corresponding diffusion suboperators.
	
	On the other hand, we consider a more flexible partitioning strategy, which will be referred to as domain decomposition splitting. First proposed by Vabishchevich in \cite{Vabishchevich1989} and further extended in \cite{Vabishchevich2008}, this technique provides a partition of the elliptic operator based on a suitable domain decomposition. The present paper focuses on the version described by \cite{MathewPolyakovRussoWang1998}.
	
	Let us consider a domain decomposition of $\Omega$ into $M$ overlapping subdomains, denoted by $\{\Omega_j\}_{j=1}^M$. Each subdomain is defined as the union of $q$ connected and pairwise disjoint components $\Omega_j = \bigcup_{i=1}^q \Omega_{j,i}$, for $j \in \{1,\dots,M\}$. The overlapping size of the subdomains is henceforth denoted by $\beta$. Figure \ref{fig:domaindecomposition} depicts the construction of the decomposition for two overlapping subdomains, such that each of them is defined as the union of $q$ vertical strips. Remarkably, each strip overlaps only with its two neighboring strips, corresponding to the opposite subdomain. A thorough description on the construction of the decomposition is detailed in \cite{ArrarasPortero2015}. Moreover, \cite{ArrarasGasparPorteroRodrigo2023}, \cite{ArrarasintHoutHundsdorferPortero2017} and \cite{PorteroArrarasJorge2010} show different domain decompositions considering four or six subdomains for various geometries.
	
	
	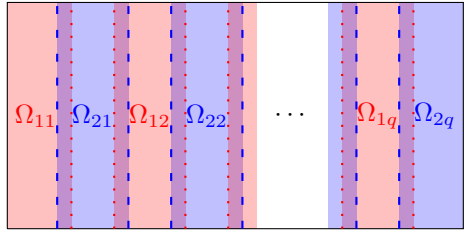
\begin{figure}[t]
		\centering
		\begin{tikzpicture}[scale=0.75]
			\fill[nearly transparent, red] (0,0) -- (1.125,0) -- (1.125,4) -- (0,4);
			\fill[nearly transparent, blue] (0.875,0) -- (2.125,0) -- (2.125,4) -- (0.875,4);
			\fill[nearly transparent, red] (1.875,0) -- (3.125,0) -- (3.125,4) -- (1.875,4);
			\fill[nearly transparent, blue] (2.875,0) -- (4.125,0) -- (4.125,4) -- (2.875,4);
			\fill[nearly transparent, red] (5.875,0) -- (7.125,0) -- (7.125,4) -- (5.875,4);
			\fill[nearly transparent, blue] (6.875,0) -- (8,0) -- (8,4) -- (6.875,4);
			\fill[nearly transparent, red] (3.875,0) -- (4.375,0) -- (4.375,4) -- (3.875,4);
			\fill[nearly transparent, blue] (5.625,0) -- (6.125,0) -- (6.125,4) -- (5.625,4);
			\draw (0,0) -- (8,0) -- (8,4) -- (0,4) -- (0,0);
			\draw[red,thick,loosely dotted] (1.125,0) -- (1.125,4);
			\draw[blue,thick,loosely dashed] (0.875,0) -- (0.875,4);
			\draw[red,thick,loosely dotted] (1.875,0) -- (1.875,4);
			\draw[blue,thick,loosely dashed] (2.125,0) -- (2.125,4);
			\draw[red,thick,loosely dotted] (3.125,0) -- (3.125,4);
			\draw[blue,thick,loosely dashed] (2.875,0) -- (2.875,4);
			\draw[red,thick,loosely dotted] (3.875,0) -- (3.875,4);
			\draw[blue,thick,loosely dashed] (4.125,0) -- (4.125,4);
			\draw[red,thick,loosely dotted] (5.875,0) -- (5.875,4);
			\draw[blue,thick,loosely dashed] (6.125,0) -- (6.125,4);
			\draw[red,thick,loosely dotted] (7.125,0) -- (7.125,4);
			\draw[blue,thick,loosely dashed] (6.875,0) -- (6.875,4);
			\path (0.5,2) node [red] {$\Omega_{11}$};
			\path (1.5,2) node [blue] {$\Omega_{21}$};
			\path (2.5,2) node [red] {$\Omega_{12}$};
			\path (3.5,2) node [blue] {$\Omega_{22}$};
			\path (5,2) node [black] {$\cdots$};
			\path (6.5,2) node [red] {$\Omega_{1q}$};
			\path (7.5,2) node [blue] {$\Omega_{2q}$};
		\end{tikzpicture}
		\caption{The overlapping domain decomposition formed by two subdomains as the union of $q$ vertical strips.}
		\label{fig:domaindecomposition}
	\end{figure}
	
	In this setting, let us define a partition of unity $\{\rho_j(\mathbf{x})\}_{j=1}^M$ subordinate to the domain decomposition under consideration. According to \cite{MathewPolyakovRussoWang1998}, the functions that form the partition must belong to $\mathcal{C}^{\infty}(\overline{\Omega})$. Moreover,  the properties $\text{supp}(\rho_j(\mathbf{x})) \subseteq \overline{\Omega}_j$, $0 \leq \rho_j(\mathbf{x}) \leq 1$, and $\sum_{k=1}^M \rho_k(\mathbf{x}) = 1$ should be fulfilled, for every $\mathbf{x} \in \overline{\Omega}$ and $j \in \{1,\dots,M\}$. We then define the functions as
	\begin{displaymath}
		\rho_j(\mathbf{x}) = \left\{ \begin{array}{l l}
			0, & \text{in } \overline{\Omega} \setminus \overline{\Omega}_j, \\
			h_j(\mathbf{x}), & \text{in } \bigcup_{k=1; k \neq j}^M (\overline{\Omega}_j \cap \overline{\Omega}_k), \\[0.5ex]
			1, & \text{in } \overline{\Omega}_j \setminus \bigcup_{k=1; k \neq j}^M (\overline{\Omega}_j \cap \overline{\Omega}_k),
		\end{array} \right.
	\end{displaymath}
	for $j \in \{1,\dots,M\}$, where $h_j(\mathbf{x})\in\mathcal{C}^{\infty}(\overline{\Omega})$ is defined such that the previous conditions hold.
	
	For the particular domain decomposition defined in Figure \ref{fig:domaindecomposition}, the following functions satisfy the stated conditions
	\begin{equation} \label{partition_unity_M2}
		\rho_1(\mathbf{x}) = \left\{ \begin{array}{l l}
			0, & \text{in } \overline{\Omega} \setminus \overline{\Omega}_1, \\
			h(x), & \text{in } \overline{\Omega}_1 \cap \overline{\Omega}_2, \\
			1, & \text{in } \overline{\Omega}_1 \setminus \overline{\Omega}_2,
		\end{array} \right.
	\end{equation}
	and $\rho_2(\mathbf{x}) = 1 - \rho_1(\mathbf{x})$, where $0 \leq h(x) \leq 1$ is a suitable trigonometric function such that $\rho_j \in \mathcal{C}^{\infty}(\overline{\Omega})$, for $j \in \{1,2\}$, and $\mathbf{x} = (x,y) \in \overline{\Omega}$.
	
	Let us then consider the following suboperators of the elliptic term as the partition
	\begin{displaymath}
		L_j u = \nabla \cdot (\rho_j D\nabla u) - \rho_j c u, \quad j \in \{1,\dots,M\}.
	\end{displaymath}
	
	\subsection{Method of lines and time-splitting integrators}
	
	Let us recall the initial-boundary value problem \cref{cont_problem}, together with the partitioned elliptic operator $L = L_1 + \cdots + L_M$. Although the present paper only considers in practice the partitioning strategies described in \Cref{susec:tyspl} for $M = 2$, we keep the general description and analysis for the case of considering an arbitrary  number $M$ of splitting terms. This generalization enables the reader to consider problems in three spatial dimensions, as well as domain decompositions with more than two subdomains.
	
	Let us first discretize problem \cref{cont_problem} in space. For that purpose, we define a suitable mesh over the spatial domain $\Omega$, hereafter denoted by $\Omega_h$. The parameter $h$ represents the maximal grid spacing of such a mesh. Then, let us consider a suitable space discretization of the problem, e.g., finite differences, finite elements, or finite volumes. The resulting semidiscrete problem is formulated as
	\begin{equation} \label{semidisc_prob}
		\left\{ \begin{array}{l l}
			U'(t) = (A_1 + \cdots + A_M)U(t) + F(t), & t \in (0,T], \\ U(0) = \mathcal{P}u_0,
		\end{array}\right.
	\end{equation}
	where $U$ and $A_i$ are the approximations of $u$ and $L_i$, respectively, for $i \in \{1,\dots,M\}$, $F$ comprises the approximation of $f$ and the boundary condition $g$, and $\mathcal{P}$ is a suitable projection or restriction operator.
	
	In order to solve the partitioned problem \cref{semidisc_prob}, we introduce suitable first-order splitting time integrators. Throughout this subsection, we consider a uniform partition of the time interval $[0,T]$ into $N$ subintervals. The set of points that forms such time mesh is denoted by $\{t_0,t_1,\cdots,t_N\}$, where $0 = t_0 < t_1 < \cdots < t_N = T$, with time step $\tau = t_{n+1} - t_n$, for every $n \in \{0,1,\dots,N-1\}$. Moreover, let us denote by $U_n$ the fully discrete solution that approximates $U(t_n)$, and define the source term $F_n = F(t_n)$, for every $n \in \{0,1,\dots,N\}$.
	
	In this framework, we propose two first-order splitting time integrators to solve the partitioned semidiscrete problem \cref{semidisc_prob}. On the one hand, we consider the fractional implicit Euler (FIE) method, described in \cite{Marchuk1990}. It is also known as the locally one-dimensional (LOD) backward Euler method (cf. \cite{HundsdorferVerwer2003}). The scheme is essentially formulated as a Lie operator splitting with integration of the resulting subproblems by the backward Euler method. Thus, it is defined as follows: given $U_0 = \mathcal{P}u_0$, for $n \in \{0,1,\dots,N-1\}$,
	\begin{equation} \label{FIE_method}
		\left\{ \begin{array}{l l}
			(I - \tau A_1)U_{n,1} = U_n + \tau F_{n+1}, \\
			(I - \tau A_j)U_{n,j} = U_{n,j-1}, & j \in \{2,\dots,M\}, \\
			U_{n+1} = U_{n,M}.
		\end{array} \right.
	\end{equation}
	The scheme is $L$-stable and first-order convergent (for more details, see \cite{Samarskii1963}). The stability function of the scheme for $M$ splitting terms is defined as
	\begin{equation} \label{stab_func_FIE}
		R_{\text{FIE}}(z_1,\dots,z_M) = \prod_{j=1}^M \frac{1}{1 - z_j},
	\end{equation}
	where $(z_1,\dots,z_M) \in \mathbb{C}^M$.
	
	On the other hand, we propose a slightly modified time-splitting integrator, the Douglas-Rachford (DR) method. First introduced in \cite{DouglasRachford1956} for $M = 2$ and $M = 3$ splitting terms, it was further generalized for an arbitrary number of terms in \cite{DouglasGunn1962} and \cite{DouglasGunn1964}. The scheme is described as follows: given $U_0 = \mathcal{P}u_0$, for $n \in \{0,1,\dots, N-1\}$,
	\begin{equation} \label{DR_method}
		\left\{ \begin{array}{l l}
			U_{n,0} = (I+\tau A)U_n + \tau F_{n+1}, \\
			(I - \tau A_j)U_{n,j} = U_{n,j-1} - \tau A_j U_n, & j \in \{1,2,\dots,M\}, \\
			U_{n+1} = U_{n,M}.
		\end{array} \right.
	\end{equation}
	The method is first-order convergent (cf. \cite{LionsMercier1979}), but, unlike the FIE scheme, it is not $L$-stable. At least, it retains the $A$-stability property for $M=2$, although only $A_0$-stability can be guaranteed for an arbitrary $M$ (see \cite{Hundsdorfer1999} for a detailed proof). Its stability function is defined as
	\begin{equation} \label{stab_func_DR}
		R_{\text{DR}}(z_1,\dots,z_M) = 1 + \left(\prod_{j=1}^M \frac{1}{1 - z_j}\right)\sum_{j=1}^M z_j.
	\end{equation}
	
	Both proposed time-splitting integrators approximate the backward Euler me-thod, although they differ in their order of splitting error. As discussed in \cite{HansenOstermannSchratz2016}, whilst the FIE scheme has first-order splitting error, the DR method shows second-order splitting error. Thus, more accurate approximations are computed by the latter, at the cost of losing the $L$-stability property.
	
	The numerical integration of parabolic problems involves solving a great number of large linear systems. Accordingly, the proposed partitioning strategies enable us to implement parallelization due to the uncoupled systems obtained for each internal stage of the resulting discrete problem. If dimensional splitting is considered, essentially one-dimensional uncoupled systems are obtained. For domain decomposition splitting, since the functions $\rho_j$ that form the partition of unity vanish outside subdomain $\Omega_j$, each internal stage consists of $q$ independent subsystems, one for each connected component.
	
	Remarkably, dimensional splitting offers stronger parallelization power, although all the data needs to be communicated among all the processors after solving each internal stage. Domain decomposition splitting, despite being more restrictive in the maximum number of processors permited, has a more efficient communication procedure, since only the data located in the overlapping regions needs to be communicated between not more than two processors.
	
	\section{The parareal algorithm for a partitioned problem}
	\label{sec:par}
	
	This section contains a brief summary of the main aspects of the parallel-in-time parareal algorithm, which will be used as the basis for our proposed solvers. This method was first introduced in \cite{LionsMadayTurinici2001}, and equivalent reformulations of the algorithm have subsequently been proposed in \cite{BafficoBernardMadayTuriniciZerah2002} and \cite{BalMaday2002}. This rather simple but efficient parallel-in-time scheme has been widely analyzed throughout literature, with promising results for parabolic problems (cf. \cite{GanderVandewalle2007}).
	
	Let us first consider a coarse uniform partition of the time interval $[0,T]$ into $N_c$ subintervals of the form $[T_{n},T_{n+1}]$, for $n \in \{0,1,\dots,N_c-1\}$, such that $\Delta T = T_{n+1} - T_n = T/N_c$ denotes the coarse step size. On the other hand, let us uniformly divide each of the intervals $[T_n,T_{n+1}]$ into $s$ smaller subintervals denoted by $[t_{ns+j},t_{ns+j+1}]$, for $j \in \{0,1,\dots,s-1\}$, where $\delta t = \Delta T/s = T/(sN_c)$ is the corresponding fine step size, such that $t_{ns} = T_n$ and $t_{(n+1)s} = T_{n+1}$.
	
	The algorithm considers two integrators, namely, the coarse and fine propagators. The former is inexpensive and inaccurate, whereas the latter is expensive but accurate. Moreover, we consider the coarse and fine partitions in time for the corresponding coarse and fine propagators, respectively. The key idea of the algorithm is to parallelize the integration of the problem by the fine propagator on each interval $[T_n,T_{n+1}]$. 
	
	Let $\mathcal{G}_{\Delta T}(T_n,T_{n+1},v)$ denote the approximation at $T_{n+1}$ obtained after a single step $\Delta T$ of the coarse propagator for a given problem, with initial value $v$ at $T_n$. Analogously, let $\mathcal{F}_{\delta t}(t_{ns+j},t_{ns+j+1},v)$ be the approximation after a step $\delta t$ computed by the fine propagator at $t_{ns+j+1}$, with initial value $v$ at $t_{ns+j}$.
	
	The algorithm usually considers the initial guess $U_0^0 = \mathcal{P}u_0$ and
	\begin{displaymath}
		U_{n+1}^0 = \mathcal{G}_{\Delta T}(T_n,T_{n+1},U_n^0), \quad n \in \{0,1,\dots,N_c-1\}.
	\end{displaymath}
	For $k \in \{0,1,2,\dots\}$, the parareal algorithm is then defined as follows:
	\begin{displaymath}
		\arraycolsep=1.4pt
		\begin{array}{ll}
			U_0^{k+1} &= \mathcal{P}u_0, \\[0.5ex]
			U_{n+1}^{k+1} &= \mathcal{G}_{\Delta T}(T_n,T_{n+1},U_n^{k+1}) + \mathcal{F}_{\delta t}^s(T_n,T_{n+1},U_n^k) - \mathcal{G}_{\Delta T}(T_n,T_{n+1},U_n^k),
		\end{array}
	\end{displaymath}
	for $n \in \{0,1,\dots,N_c-1\}$. In the previous expression, we use the compact notation $\mathcal{F}_{\delta t}^s(T_n,T_{n+1},U_n^k)$ for the approximation $\tilde{U}_{(n+1)s}$, resulting from the recurrence relation
	\begin{displaymath}
		\tilde{U}_{ns+j+1} = \mathcal{F}_{\delta t}(t_{ns+j},t_{ns+j+1},\tilde{U}_{ns+j}), \quad j \in \{0,1,\dots,s-1\},
	\end{displaymath}
	with $\tilde{U}_{ns} = U_n^k$. Notice that $\mathcal{F}_{\delta t}^s(T_n,T_{n+1},U_n^k)$ can be computed in parallel for every $n \in \{0,1,\dots,N_c-1\}$, since the values $U_n^k$ are known from the previous iteration. Therefore, up to $N_c$ processors can be devoted to the parallel implementation of the algorithm.
	
	Among the different interpretations of the parareal algorithm, we highlight that it is described as a predictor-corrector method in \cite{StaffRonquist2005}, whereas it was proven to belong to the family of multiple shooting methods (see \cite{GanderVandewalle2007} for further details). Finally, from a multigrid point of view, it can be seen as a two-level multigrid in time method (cf. \cite{FalgoutFriedhoffKolevMaclachlanSchroder2014}).
	
	In the sequel, the fine approximation stands for the approximation computed sequentially by the fine propagator with step size $\delta t$. Even though the method is iterative, it converges to the fine approximation in a finite number of iterations (cf. \cite{GanderKwokZhang2018}, Theorem 6). In fact, $N_c$ iterations are required at most to obtain the fine approximation. However, in order to be competitive with sequential computing, the algorithm should provide a sufficiently accurate approximation in only a few iterations.
	
	Regarding stability, based on a comprehensive analysis performed by Friedhoff and Southworth in \cite{FriedhoffSouthworth2021}, it is concluded that $L$-stability is a key property for the coarse propagator so as to achieve convergence of the parareal algorithm. This conclusion is consistent with Theorem 2 in \cite{StaffRonquist2005}.
	
	In order to analyze the convergence properties of the parareal algorithm for a partitioned problem, let us adapt the definition of the convergence factor, proposed in \cite{GanderVandewalle2007}. We then consider the partitioned Dahlquist test problem
	\begin{equation} \label{part_Dahlquist}
		\left\{ \begin{array}{l l}
			y'(t) = (\lambda_1 + \cdots + \lambda_M)y(t), & t \in (0,T], \\
			y(0) = y_0,
		\end{array} \right.
	\end{equation}
	for $M$ splitting terms and with $\lambda_i \in \mathbb{C}$, for $i \in \{1,\dots,M\}$ (see Chapter IV in \cite{HundsdorferVerwer2003}). This generalization is especially useful in our framework, on the grounds that the elliptic operator is partitioned following the splitting techniques described in \Cref{sec:split}.
	
	In this setting, the convergence factor of the parareal algorithm is defined as
	\begin{equation} \label{conv_factor}
		\mathcal{K}(z_1,\dots,z_M,s) = \frac{|R_{\mathcal{F}}(z_1/s,\dots,z_M/s)^s - R_{\mathcal{G}}(z_1,\dots,z_M)|}{1 - |R_{\mathcal{G}}(z_1,\dots,z_M)|},
	\end{equation}
	where $R_{\mathcal{G}}$ and $R_{\mathcal{F}}$ are the stability functions of the coarse and fine propagators, respectively, and $s = \Delta T/\delta t$. Note that for $M = 1$, the convergence factor \cref{conv_factor} is equivalent to the formulation presented in \cite{GanderVandewalle2007}. Let us denote by $y_n^k$ the approximation of the parareal algorithm for problem \cref{part_Dahlquist} at $t = T_n$ and $k$-th iteration. Similarly, let $y_n^F$ denote the fine approximation for problem \cref{part_Dahlquist} at $t = T_n$. Then, the following inequality is fulfilled for every $k \in \{0,1,2,\dots\}$ (cf. \cite{HundsdorferVerwer2003})
	\begin{displaymath}
		\sup_{n>0} |y_n^F - y_n^k| \leq \left(\mathcal{K}(\lambda_1\Delta T, \dots, \lambda_M\Delta T, s)\right)^k \sup_{n>0} |y_n^F - y_n^0|.
	\end{displaymath}
	Therefore, if the bound $\mathcal{K}(\lambda_1\Delta T, \dots, \lambda_M\Delta T, s) < 1$ holds, the parareal algorithm converges to the fine approximation at each iteration. Remarkably, more demanding bounds on the convergence factor are considered in literature, as we further detail in the following section (cf. \cite{Wu2015}).
	
	\begin{remark}
		The previous results can be generalized for matrix differential problems of the form
		\begin{displaymath}
			\left\{ \begin{array}{l l}
				y'(t) = (A_1 + \cdots + A_M)y(t), & t \in (0,T], \\
				y(0) = y_0,
			\end{array} \right.
		\end{displaymath}
		with $A_i \in \mathbb{R}^{m \times m}$, for $i \in \{1,\dots,M\}$, if the set $\{A_1,\dots,A_M\}$ is simultaneously diagonalizable. Thus, assuming that there exists an invertible matrix $V \in \mathbb{R}^{m \times m}$ such that $V^{-1}A_iV$ is a diagonal matrix, for $i \in \{1,\dots,M\}$, the following inequality holds
		\begin{displaymath}
			\sup_{n>0} \|Ve_n^k\|_{\infty} \leq \left(\max_{\substack{\lambda_i \in \sigma(A_i) \\ i \in \{1,\dots,M\}}} \mathcal{K}(\lambda_1\Delta T, \dots, \lambda_M\Delta T, s)\right)^k \sup_{n>0} \|Ve_n^0\|_{\infty},
		\end{displaymath}
		where $e_n^k = y_n^F - y_n^k$, and $\sigma(A_i)$ stands for the set of eigenvalues of $A_i$, for $i \in \{1,\dots,M\}$. The imposed conditions are really demanding, and one cannot apply such theoretical results directly, for instance, for domain decomposition splittings. However, as stated in \cite{BuvoliMinion2024}, the qualitative behaviour of the method and its theoretical convergence analysis are closely related in practice. Numerical experiments performed in \Cref{sec:num} also validate this claim. Therefore, we will consider only the scalar problem \cref{part_Dahlquist} for the convergence analysis of the methods.
	\end{remark}
	
	\section{Formulation and convergence analysis of the solvers}
	\label{sec:conv}
	
	In this section, we introduce the proposed space-time parallel solvers, together with their main theoretical results, which are stated and proven. The major contribution of this work is to consider splitting time integrators as the propagators of the parareal algorithm, resulting in methods that permit parallelization in both space and time. It is particularly noteworthy that, unlike the classical parareal algorithm, this version permits parallelization of the coarse propagator. This is a key advantage, since the use of all the available processors is optimized.
	
	As mentioned in the previous section, the coarse propagator is required to be $L$-stable. We hence select the FIE method \cref{FIE_method} for that purpose. On the other hand, both the aforementioned scheme \cref{FIE_method} and the DR method \cref{DR_method} are separately considered as the fine propagators. Therefore, we present the following space-time parallel methods:
	\begin{itemize}
		\item the Parareal/FIE-FIE method: the parareal algorithm with the FIE method \cref{FIE_method} as both propagators; and
		\item the Parareal/FIE-DR method: the parareal algorithm with the FIE method \cref{FIE_method} as the coarse propagator and the DR method \cref{DR_method} as the fine propagator.
	\end{itemize}
	
	\subsection{Convergence analysis of the Parareal/FIE-FIE method}
	
	We first show that the convergence factor \cref{conv_factor} of the Parareal/FIE-FIE method is bounded by $1/3$ if $z_i \in \mathbb{R}^-$, for every $i \in \{1,\dots,M\}$. This more demanding condition is usually required in order to ensure a faster convergence of the parareal iterative procedure (see, e.g., \cite{Wu2015}, \cite{Wu2016}, and \cite{WuZhou2015}). We first prove the bound for two splitting terms, and, subsequently, we generalize the result for $M$ arbitrary terms.
	
	\begin{theorem} \label{thm:FIE-FIE}
		Let $\mathcal{K}_{\text{FIE-FIE}}(z_1,z_2,s)$ be the convergence factor \cref{conv_factor} of the Para-real/FIE-FIE method for $M = 2$. Then, the bound $\mathcal{K}_{\text{FIE-FIE}}(z_1,z_2,s) \leq 1/3$ holds for every $(z_1,z_2) \in (-\infty,0)^2$ and $s \geq 1$.
	\end{theorem}
	\begin{proof}
		Let $R_{\text{FIE}}(z_1,z_2)$ denote the stability function \cref{stab_func_FIE} for $M = 2$. We need to prove the inequality
		\begin{displaymath}
			\mathcal{K}_{\text{FIE-FIE}}(z_1,z_2,s) = \frac{|(R_{\text{FIE}}(z_1/s,z_2/s))^s - R_{\text{FIE}}(z_1,z_2)|}{1 - |R_{\text{FIE}}(z_1,z_2)|} \leq \frac{1}{3},
		\end{displaymath}
		for $(z_1,z_2) \in (-\infty,0)^2$, or equivalently, since $0 < R_{\text{FIE}}(z_1,z_2) < 1$,
		\begin{equation} \label{FIE_cond1}
			(R_{\text{FIE}}(z_1/s,z_2/s))^s \leq \frac{1}{3} + \frac{2}{3}R_{\text{FIE}}(z_1,z_2)
		\end{equation}
		and
		\begin{equation} \label{FIE_cond2}
			(R_{\text{FIE}}(z_1/s,z_2/s))^s \geq -\frac{1}{3} + \frac{4}{3}R_{\text{FIE}}(z_1,z_2).
		\end{equation}
		
		Defining $G(s) = (1-z/s)^{-s}$ for any fixed $z \in (-\infty,0)$, it is straightforward to prove that $G$ is monotonically decreasing. Then, since $G(s) > 0$ for every $s \geq 1$, the inequality $(R_{\text{FIE}}(z_1/{s_1},z_2/{s_1}))^{s_1} > (R_{\text{FIE}}(z_1/{s_2},z_2/{s_2}))^{s_2}$ holds for $s_2 > s_1 \geq 1$.
		
		Given that scheme \cref{FIE_method} is $L$-stable, $R_{\text{FIE}}(z_1,z_2) < 1$ is satisfied for every $(z_1,z_2)$ $ \in (-\infty,0)^2$, and, as a consequence,
		\begin{displaymath}
			R_{\text{FIE}}(z_1,z_2) = \frac{1}{3} R_{\text{FIE}}(z_1,z_2) + \frac{2}{3}R_{\text{FIE}}(z_1,z_2) < \frac{1}{3} + \frac{2}{3} R_{\text{FIE}}(z_1,z_2).
		\end{displaymath}
		Finally, since $(R_{\text{FIE}}(z_1/s,z_2/s))^s \leq R_{\text{FIE}}(z_1,z_2)$ holds for every $s \geq 1$, condition \cref{FIE_cond1} is fulfilled.
		
		On the other hand, since $\displaystyle \lim_{s \rightarrow \infty} (R_{\text{FIE}}(z_1/s,z_2/s))^s = e^{z_1+z_2}$, $(R_{\text{FIE}}(z_1/s,z_2/s))^s > e^{z_1+z_2}$ holds for every $s \geq 1$, and, in order to show condition \cref{FIE_cond2}, it suffices to prove
		\begin{equation} \label{FIE_cond3}
			e^{z_1+z_2} \geq - \frac{1}{3} + \frac{4}{3(1-z_1)(1-z_2)}.
		\end{equation}
		Note that the right-hand side term in \cref{FIE_cond3} is negative for every $(z_1,z_2) \in (-\infty,0)^2\setminus [-3,0)^2$. Thus, we only need to prove \cref{FIE_cond3} in the domain $[-3,0)^2$. Rearranging the expression, it is sufficient to show
		\begin{equation} \label{FIE_cond4}
			H(z_1,z_2) = (1-z_1)(1-z_2)(3e^{z_1+z_2} + 1) \geq 4,
		\end{equation}
		for $(z_1,z_2) \in [-3,0)^2$. A simple analysis of the absolute extrema of $H$ in $[-3,0]^2$ yields that the only candidates to absolute extrema are the points $(0,0)$, $(-3,0)$, $(0,-3)$, $(-3,-3)$, $(z^*,0)$, $(0,z^*)$, $(z^{**},0)$, and $(0,z^{**})$, where $z^*$ and $z^{**}$ are the only two real solutions of the equation $3ze^z + 1 =0$ on the interval $[-3,0]$.
		
		Evaluating function $H$ at these points, one can conclude that
		\begin{displaymath}
			4 = H(0,0) \leq H(z_1,z_2) \leq H(-3,-3) = 16\left(1+3/e^6\right)
		\end{displaymath}
		for every $(z_1,z_2) \in [-3,0]^2$. Therefore, inequality \cref{FIE_cond4} holds and the claim is proven.
	\end{proof}
	
	The previous result can be generalized for $M$ arbitrary splitting terms, as it is shown in the following corollary.
	
	\begin{corollary}
		Let $\mathcal{K}_{\text{FIE-FIE}}(z_1,\dots,z_M,s)$ be the convergence factor \cref{conv_factor} of the Parareal/FIE-FIE method, for $M$ arbitrary splitting terms, with $M \geq 2$. Then, the bound $\mathcal{K}_{\text{FIE-FIE}}(z_1,\dots,z_M,s) \leq 1/3$ is satisfied for every $(z_1,\dots,z_M) \in (-\infty,0)^M$ and $s \geq 1$.
	\end{corollary}
	\begin{proof}
		Following the proof of \Cref{thm:FIE-FIE} until \cref{FIE_cond4} considering $M$ terms, it yields that the claim is fulfilled if the following bound holds
		\begin{equation} \label{FIE_cond5}
			H_M(z_1,\dots,z_M) =\left( \prod_{j=1}^M (1-z_j)\right) \left(3e^{\sum_{j=1}^M z_j} + 1\right) \geq 4,
		\end{equation}
		for every $(z_1,\dots,z_M) \in [-3,0]^M$.
		
		Let us prove inequality \cref{FIE_cond5} by mathematical induction on $M$. Firstly, the inequality was proven in \Cref{thm:FIE-FIE} for $M = 2$. Furthermore, let us assume that inequality \cref{FIE_cond5} holds for $M = \tilde{M}-1$, with $\tilde{M} \geq 3$, and let us show that it is also satisfied for $M = \tilde{M}$.
		
		Let us find the candidates to relative extrema of $H_{\tilde{M}}$ in the domain $(-3,0)^{\tilde{M}}$. Imposing
		\begin{displaymath}
			\frac{\partial H_{\tilde{M}}}{\partial z_i} = \left(\prod_{\substack{j=1 \\ j \neq i}}^{\tilde{M}} (1-z_j)\right) \left(-1-3z_i e^{\sum_{j=1}^{\tilde{M}} z_j}\right) = 0, 
		\end{displaymath}
		for $i \in \{1,\dots, \tilde{M}\}$, it follows that the inequalities
		\begin{displaymath}
			-1-3z_ie^{\sum_{j=1}^{\tilde{M}} z_j} = 0
		\end{displaymath}
		must be satisfied. Subtracting the expressions pairwise, we obtain the necessary conditions $z_i = z_j$  for every $i,j \in \{1,\dots,\tilde{M}\}$ for the candidates to relative extrema of function $H_{\tilde{M}}$.
		
		Since $-1 -3ze^{\tilde{M}z} = 0$ does not have any solution for $z \in [-3,0]$ and $\tilde{M} \geq 3$, it follows that the only candidates to extrema of $H_{\tilde{M}}$ are located on the boundaries of dimension $\tilde{M}-1$ of the compact set $[-3,0]^{\tilde{M}}$. In particular, on the part of the boundary where $z_i = 0$, $H_{\tilde{M}}|_{z_i=0}$ is essentially the function $H_{\tilde{M}-1}$. Then, applying the induction hypothesis, it follows that $H_{\tilde{M}-1}$ is bounded from below by 4, or equivalently, $H_{\tilde{M}}(z_1,\dots,z_{\tilde{M}}) \geq 4$, for $(z_1,\dots,z_{\tilde{M}}) \in [-3,0]^{\tilde{M}}$ and $z_i = 0$, $i \in \{1,\dots,\tilde{M}\}$.
		
		In addition, on the part of the boundary where $z_i = -3$ (without loss of generality, we consider $z_{\tilde{M}} = -3$), it holds
		\begin{displaymath}
			H_{\tilde{M}}(z_1,\dots,z_{\tilde{M}-1},-3) = 4\left(\prod_{j=1}^{\tilde{M}-1}(1-z_j)\right) \left(3e^{-3}e^{\sum_{j=1}^{\tilde{M}-1}z_j} + 1\right) \geq 4,
		\end{displaymath}
		for $(z_1,\dots,z_{\tilde{M}-1}) \in [-3,0]^{\tilde{M}-1}$, since all the factors in the product accompanying the number 4 are greater than or equal to 1. Therefore, inequality \cref{FIE_cond5} is fulfilled for every $(z_1,\dots,z_M) \in [-3,0]^M$.
	\end{proof}
	
	\Cref{fig:FIE_conv_fact} depicts $\mathcal{K}_{\text{FIE-FIE}}$ over the real negative axis, for two splitting terms and different values of the parameter $s$. No qualitative difference can be observed among the different values of $s$, except for a slight increase of the maximum value. \Cref{fig:FIE_conv_fact} also shows the convergence region of the Parareal/FIE-FIE method in the complex plane. The solver is thus convergent in almost the entire complex left half-plane, except for the region near the imaginary axis.
	
	\begin{figure}[t]
		\centering
		\subfloat[$\mathcal{K}_{\text{FIE-FIE}}$ over the real negative axis]{\label{conv_FIE-FIE}\includegraphics[width=0.465\textwidth]{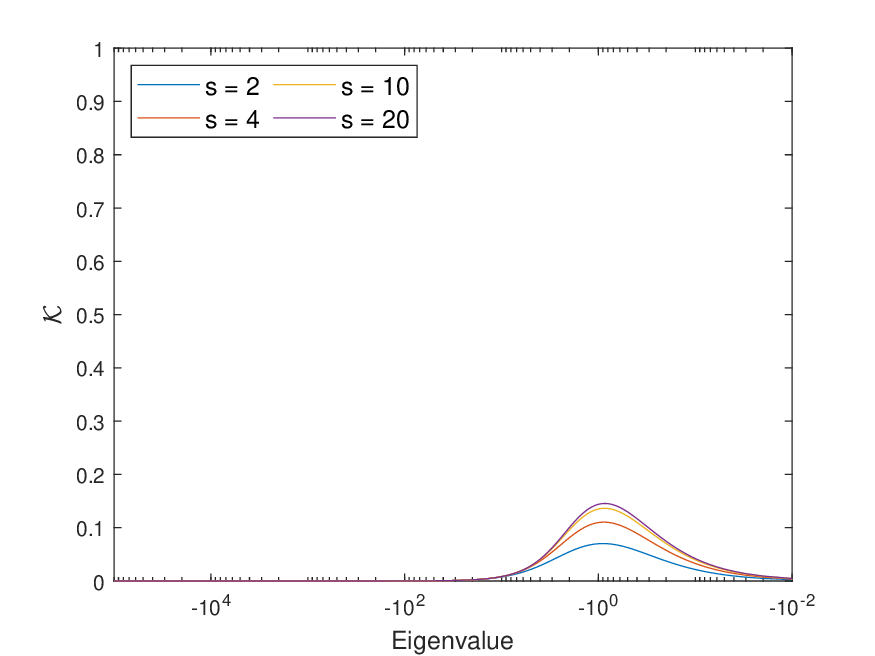}}
		\subfloat[Convergence region in $(-1,0) \times (-20,20)$ ]{\label{FIE-FIE_region}\includegraphics[width=0.465\textwidth]{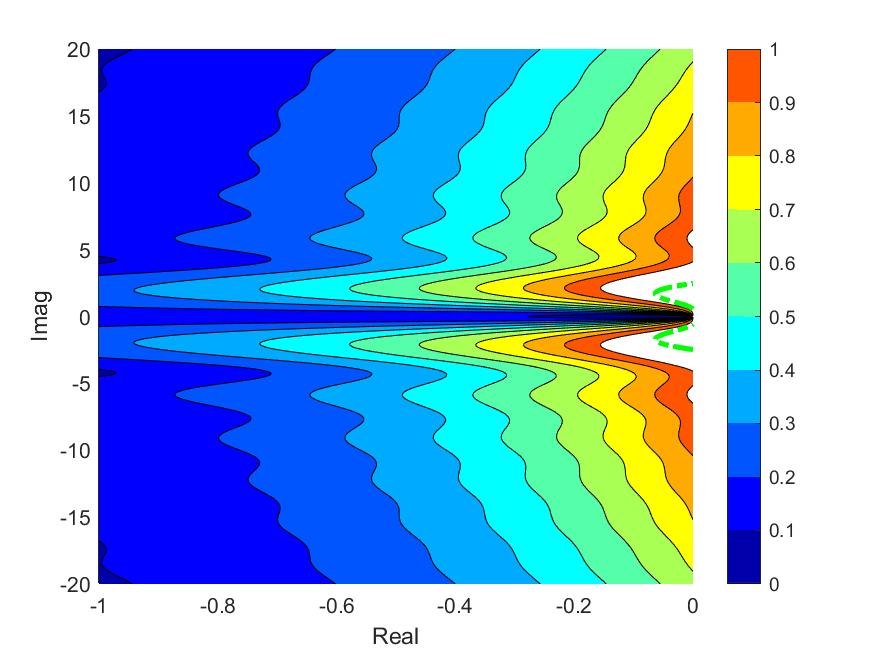}}
		
		\subfloat[Convergence region in $(-2.5\cdot 10^7,0) \times (-10^7,10^7)$]{\label{FIE-FIE_region_huge}\includegraphics[width=0.465\textwidth]{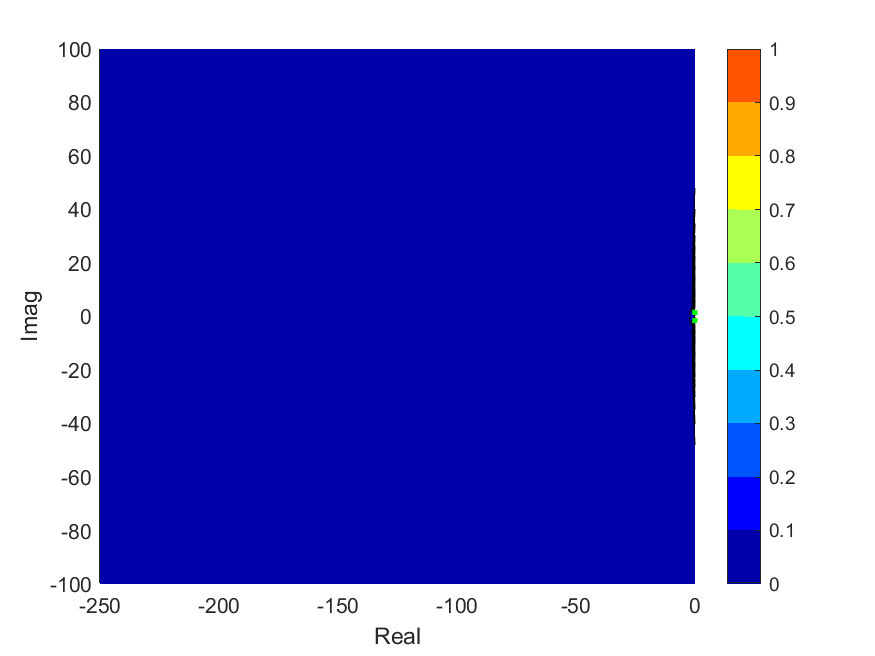}}
		\caption{Convergence factor $\mathcal{K}_{\text{FIE-FIE}}$ over the real negative axis for changing value of $s$, and convergence region over the complex left half-plane in the domains $(-1,0) \times (-20,20)$ and $(-2.5\cdot 10^7,0) \times (-10^7,10^7)$, for $s = 1000$ and for $s = 20$ delimited by a dashed green line. In both cases two equal splitting terms are considered.}
		\label{fig:FIE_conv_fact}
	\end{figure}
	
	\subsection{Convergence analysis of the Parareal/FIE-DR method}
	
	The convergence factor of the Parareal/FIE-DR method can only be bounded by 1 on the real negative axis. In particular, the factor is close to 1 for eigenvalues with large real negative part, as we graphically show at the end of this subsection. This is caused by lack of $L$-stability of the fine propagator.
	
	\begin{theorem} \label{thm:FIE-DR}
		Let $\mathcal{K}_{\text{FIE-DR}}(z_1,z_2,s)$ be the convergence factor \cref{conv_factor} of the Para-real/FIE-DR method, for $M = 2$. Then, the bound $\mathcal{K}_{\text{FIE-DR}}(z_1,z_2,s) < 1$ is fulfilled, for every $(z_1,z_2) \in (-\infty,0)^2$ and $s \geq 1$.
	\end{theorem}
	\begin{proof}
		The bound $\mathcal{K}_{\text{FIE-DR}}(z_1,z_2,s) < 1$, for $(z_1,z_2) \in (-\infty,0)^2$ and $s \geq 1$, is equivalent to the inequalities $(R_{\text{DR}}(z_1/s,z_2/s))^s < 1$ and
		\begin{equation} \label{DR_cond1}
			(R_{\text{DR}}(z_1/s,z_2/s))^s > 2R_{\text{FIE}}(z_1,z_2) - 1,
		\end{equation}
		where $R_{\text{FIE}}$ and $R_{\text{DR}}$ are the stability functions \cref{stab_func_FIE} and \cref{stab_func_DR} for $M = 2$, respectively.
		
		The first condition is trivially fulfilled, owing to the $A$-stability of the DR method. Moreover, for fixed $(z_1,z_2) \in (-\infty,0)^2$, the function $G(s) = (R_{\text{DR}}(z_1/s,z_2/s))^s$ is monotonically decreasing, since $G'(s) < 0$, for every $s \geq 1$. Then, in the light of $\displaystyle \lim_{s \rightarrow \infty} (R_{\text{DR}}(z_1/s,z_2/s))^s = e^{z_1+z_2}$, the inequality $(R_{\text{DR}}(z_1/s,z_2/s))^s > e^{z_1+z_2}$ holds for $s \geq 1$, and, in order to show \cref{DR_cond1}, it suffices to prove
		\begin{equation} \label{DR_cond2}
			e^{z_1+z_2} > 2R_{FIE}(z_1,z_2) - 1 = \frac{2}{(1-z_1)(1-z_2)} - 1.
		\end{equation}
		Note that the right-hand side of inequality \cref{DR_cond2} is negative for every $(z_1,z_2) \in (-\infty,0)^2\setminus [-1,0)^2$. Hence, we must show that \cref{DR_cond2} is satisfied in the domain $[-1,0)^2$. Rearranging \cref{DR_cond2}, it follows that it is sufficient to prove that
		\begin{equation} \label{DR_cond3}
			H(z_1,z_2) = (1-z_1)(1-z_2)(e^{z_1+z_2} + 1) > 2,
		\end{equation}
		for every $(z_1,z_2) \in [-1,0)^2$.
		
		By a simple analysis of the absolute extrema of the function $H$ in the domain $[-1,0]^2$, it follows that $(0,0)$, $(-1,0)$, $(0,-1)$, and $(-1,-1)$ are the only candidates to absolute extrema in the compact domain $[-1,0]^2$. Evaluating function $H$ at these points, one can conclude that
		\begin{displaymath}
			2 = H(0,0) < H(z_1,z_2) \leq H(-1,-1) = 4\left(1+1/e^2\right),
		\end{displaymath}
		for every $(z_1,z_2) \in [-1,0)^2$. As a consequence, inequality \cref{DR_cond2} is satisfied.
	\end{proof}
	
	Once again, the result can be extended to the case of considering $M$ arbitrary splitting terms by using mathematical induction, as demonstrated below.
	
	\begin{corollary}
		Let $\mathcal{K}_{\text{FIE-DR}}(z_1,\dots,z_M,s)$ be the convergence factor \cref{conv_factor} of the Parareal/FIE-DR method for $M$ arbitrary splitting terms, with $M \geq 2$. Then, the bound $\mathcal{K}_{\text{FIE-DR}}(z_1,\dots,z_M,s) < 1$ is fulfilled for every $(z_1,\dots,z_M) \in (-\infty,0)^M$ and $s \geq 1$.
	\end{corollary}
	\begin{proof}
		Considering $M$ terms and following the proof of \Cref{thm:FIE-DR} until \cref{DR_cond3}, it yields that the claim holds if the following inequality is satisfied
		\begin{equation} \label{DR_cond4}
			H_M(z_1,\dots,z_M) = \left(\prod_{j=1}^M (1-z_j) \right)\left(e^{\sum_{j=1}^M z_j} + 1 \right) > 2,
		\end{equation}
		for every $(z_1,\dots,z_M) \in [-1,0)^M$. Let us prove inequality \cref{DR_cond4} by mathematical induction on $M$. First, by \Cref{thm:FIE-DR}, we conclude that, for $M=2$, \cref{DR_cond4} holds. 
		
		Moreover, let us assume that inequality \cref{DR_cond4} is satisified for $M = \tilde{M} - 1$, with $\tilde{M} \geq 3$. Now, considering $M = \tilde{M}$, we look for the points $(z_1,\ldots,z_{\tilde{M}})\in (0,1)^{\tilde{M}}$ that satisfy
		\begin{equation} \label{DR_cond5}
			\frac{\partial H_{\tilde{M}}}{\partial z_i} = -\prod_{\substack{j=1 \\ j \neq i}}^{\tilde{M}} (1-z_j) \left(1 + z_ie^{\sum_{j=1}^{\tilde{M}} z_j}\right) = 0
		\end{equation}
		for every $i \in \{1,\dots,\tilde{M}\}$. Equations (\ref{DR_cond5}) are equivalent to
		\begin{displaymath}
			1 + z_i e^{\sum_{j=1}^{\tilde{M}} z_j} = 0,
		\end{displaymath}
		for $i \in \{1,\dots,\tilde{M}\}$. Subtracting pairwise the previous equalities, it follows that $z_i = z_j$, for every $i,j \in \{1,\dots,\tilde{M}\}$. Since $1 + ze^{\tilde{M}z} = 0$ does not have any solution for $z \in [-1,0]$ and $M \geq 3$, we conclude that no point in the interior of $[-1,0)^{\tilde{M}}$ satisfies \cref{DR_cond5}, and, thus, there are not candidate to extrema in the interior of such domain.
		
		Furthermore, let us look for candidates to absolute extrema on the boundary of the domain $[-1,0]^{\tilde{M}}$. There are two types of hyper-surfaces which form that boundary, i.e., $z_i = 0$ or $z_i = -1$, for some $i \in \{1,\dots,\tilde{M}\}$. First, if $z_i = 0$, $H_{\tilde{M}}|_{z_i = 0}$ is essentially the function $H_{\tilde{M}-1}$, and by the induction hypothesis, $H_{\tilde{M}-1}(z_1,\dots,z_{\tilde{M}-1}) > 2$, for every $(z_1,\dots,z_{\tilde{M}-1}) \in [-1,0)^{\tilde{M}-1}$. Hence, $H_{\tilde{M}}|_{z_i=0}$ is bounded from below by 2.
		
		On the other hand, if $z_i = -1$ (without loss of generality, we consider $z_{\tilde{M}} = -1$),
		\begin{displaymath}
			H_{\tilde{M}}(z_1,\dots,z_{\tilde{M}-1},-1) = 2\prod_{j=1}^{\tilde{M}-1} (1-z_j) \left(e^{-1}e^{\sum_{j=1}^{\tilde{M}-1} z_j} + 1\right) > 2
		\end{displaymath}
		holds trivially, for $(z_1,\dots,z_{\tilde{M}-1}) \in [-1,0)^{\tilde{M}-1}$. Thus, \cref{DR_cond4} holds in $[-1,0)^{\tilde{M}}$ and the claim is proven.
	\end{proof}
	
	\Cref{fig:DR_conv_fact} shows the convergence factor $\mathcal{K}_{\text{FIE-DR}}$ for two equal splitting terms on the real negative axis and in the complex plane, for different values of $s$. Notice that, for large real negative eigenevalues, $\mathcal{K}_{\text{FIE-DR}}$ is close to 1, which may imply a deterioration in terms of convergence. Remarkably, for greater values of $s$, the interval where $\mathcal{K}_{\text{FIE-DR}}$ is close to zero increases. Thus, a sufficiently large value of $s$ is highly recommended in order to avoid deterioration of convergence, as we illustrate in the following section with numerical experiments. Then, the eigenvalues of the problem are embedded in the interval where $\mathcal{K}_{\text{FIE-DR}}$ is close to zero. As for the Parareal/FIE-FIE method, eigenvalues near the imaginary axis may suffer from lack of convergence.
	
	\begin{figure}[t]
		\centering
		\subfloat[$\mathcal{K}_{\text{FIE-DR}}$ over the real negative axis]{\label{conv_FIE-DR}\includegraphics[width=0.465\textwidth]{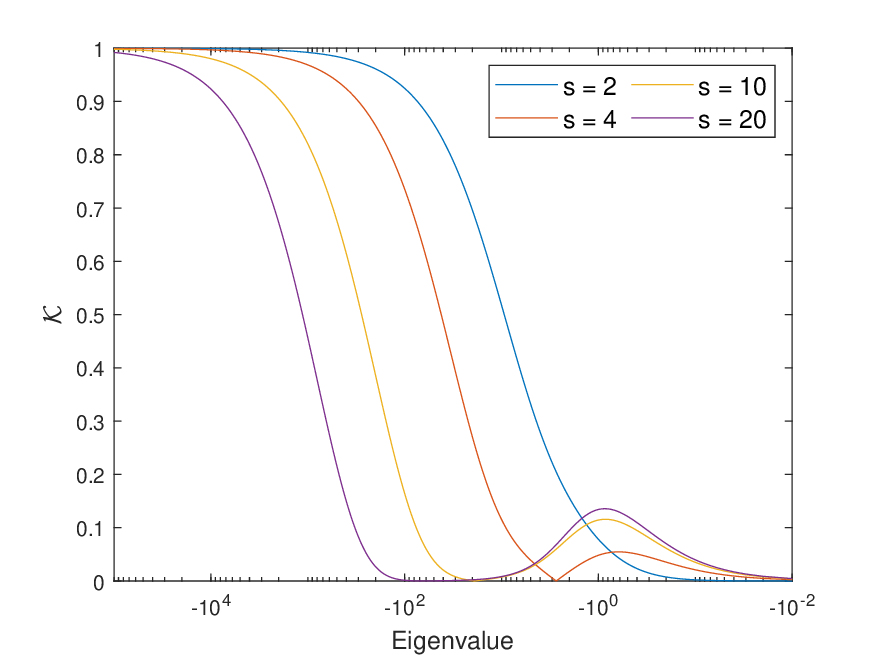}}
		\subfloat[Convergence region in $(-1,0) \times (-20,20)$]{\label{FIE-DR_region}\includegraphics[width=0.465\textwidth]{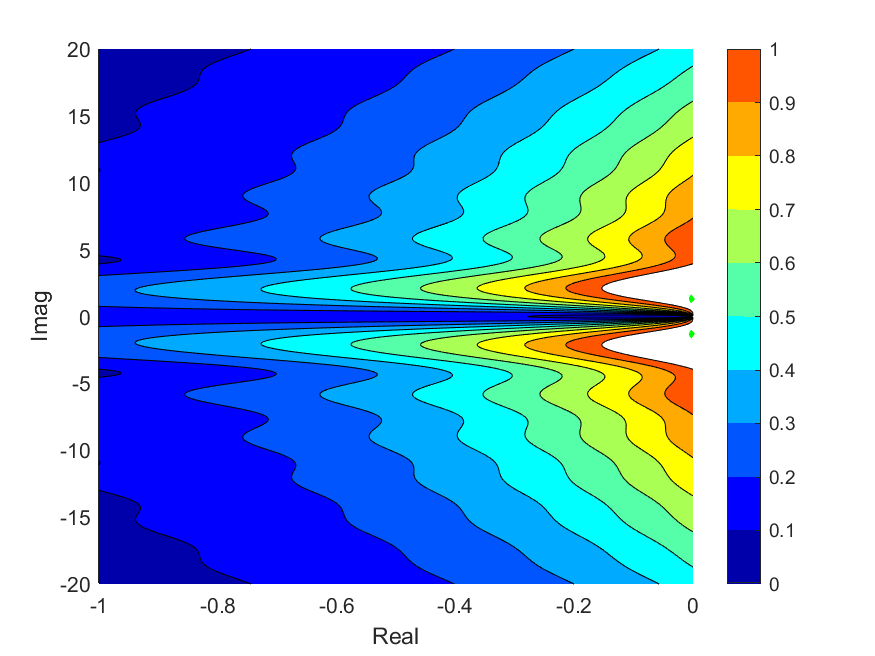}}
		
		\subfloat[Convergence region in $(-2.5\cdot 10^7,0) \times (-10^7,10^7)$]{\label{FIE-DR_region_huge}\includegraphics[width=0.465\textwidth]{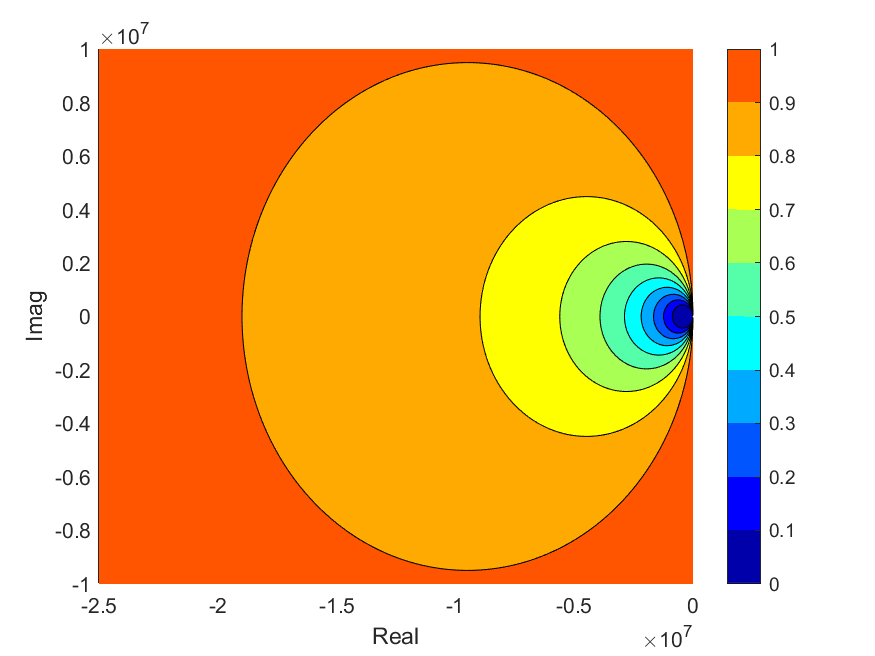}}
		\caption{Convergence factor $\mathcal{K}_{\text{FIE-DR}}$ over the real negative axis for changing value of $s$, and convergence region over the complex left half-plane in the domains $(-1,0) \times (-20,20)$ and $(-2.5\cdot 10^7,0) \times (-10^7,10^7)$, for $s = 1000$ and for $s = 20$ delimited by a dashed green line. In both cases two equal splitting terms are considered.}
		\label{fig:DR_conv_fact}
	\end{figure}
	
	\section{Numerical results}
	\label{sec:num}
	
	The current section is devoted to perform certain numerical test simulations that illustrate the theoretical analysis developed for the newly proposed space-time parallel solvers in \Cref{sec:conv}. For that purpose, we consider a rather academic problem, the heat equation with Dirichlet homogeneous boundary conditions
	\begin{equation} \label{heat_prob}
		\left\{ \begin{array}{l l}
			u_t = \nabla \cdot (D \nabla u) + f, & \text{in } \Omega \times (0,T], \\
			u = 0, & \text{on } \Gamma \times (0,T], \\
			u = 0, & \text{in } \overline{\Omega} \times \{0\},
		\end{array} \right.
	\end{equation}
	posed in the spatial domain $\Omega = (0,1)\times (0,1)$, with diffusion tensor
	\begin{displaymath}
		D = \begin{bmatrix} d_{11}(\mathbf{x}) & d_{12}(\mathbf{x}) \\ d_{12}(\mathbf{x}) & d_{22}(\mathbf{x}) \end{bmatrix},
	\end{displaymath}
	and source term $f = f(\mathbf{x},t)$. A suitable $f$ is chosen such that 
	\begin{displaymath}
		u(\mathbf{x},t) = u(x,y,t) = \sin(2\pi t)\sin(2\pi x) \sin(2\pi y)
	\end{displaymath}
	is the exact solution of problem \cref{heat_prob}. We consider second-order finite differences on a regular meshgrid for the spatial discretization, denoting the mesh size parameter by $h$. Moreover, we set $T = 1$, and the errors are computed using the maximum norm in time and the discrete 2-norm in space.
	
	In order to reliably illustrate the performance of the solvers, two different situations are considered. On the one hand, we consider a diagonal and constant diffusion tensor $D$, that is, $d_{12} = 0$, and $d_{11}$ and $d_{22}$ independent of $\mathbf{x} \in \Omega$. On the other hand, a more challenging problem is considered, where $D$ is a full tensor matrix. In this case, mixed partial derivatives appear in the elliptic operator and, as a consequence, only domain decomposition splittings can be directly implemented (see \Cref{susec:tyspl}).
	
	\subsection{Experiments with a constant diagonal diffusion tensor}
	
	In this subsection, let us consider $d_{11} = 1 = d_{22}$ and $d_{12} = 0$. In the sequel, unless the opposite is mentioned, we set the discretization parameters $h = 10^{-3}$, $\Delta T = 5\cdot 10^{-2}$, and $\delta t = 2.5\cdot 10^{-3}$, and we define a domain decomposition splitting formed by two subdomains, based on \Cref{fig:domaindecomposition}, and partition of unity \cref{partition_unity_M2} for both propagators. Two vertical strips and overlapping size $\beta = 1/16$ are considered for such decomposition.
	
	Let us denote the parareal algorithm with the implicit Euler as both propagators by the Parareal/IE-IE method. Then, \Cref{fig:error_wrt_fine} shows the evolution of the error of the approximations computed by the presented methods with respect to the fine approximation, with increasing number of iterations. We can conclude that, depending on the partitioning strategy, the slope of the error changes. Domain decomposition splitting performs slightly slower than the Parareal/IE-IE method in terms of convergence. On the other hand, dimensional splitting even improves the convergence behaviour of the unsplit method, reducing the number of iterations and cost per iteration at the same time. The two space-time parallel methods perform similarly, according to \Cref{fig:error_wrt_fine}.
	
	\begin{figure}[t]
		\centering
		\includegraphics[width=0.5\textwidth]{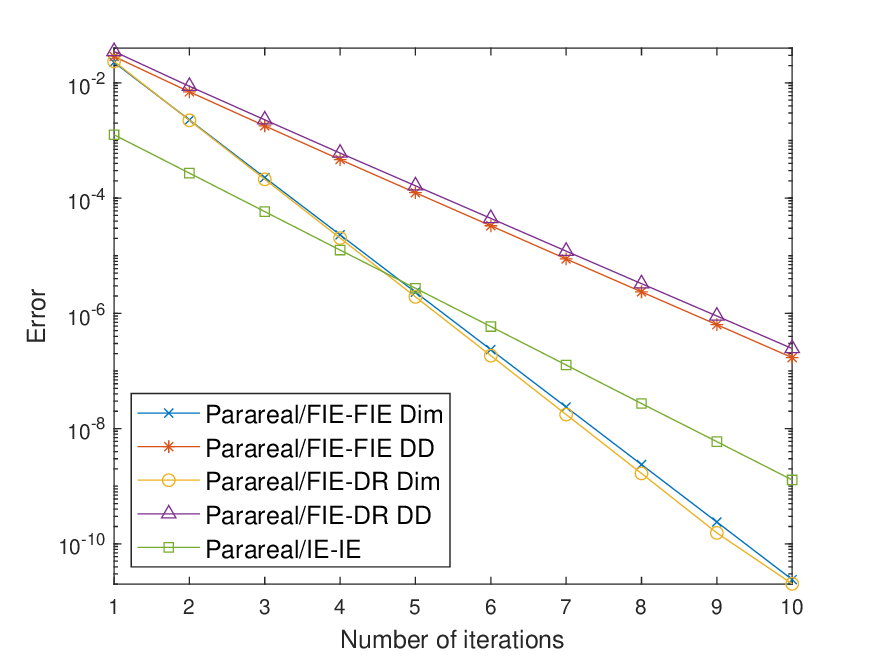}
		\caption{Evolution of the error of the space-time parallel methods and the Parareal/IE-IE method with respect to the fine approximation for the test problem \cref{heat_prob} ($d_{11} = 1 = d_{22}$, $d_{12} = 0$) for an increasing number of iterations. We set $T = 1$, $h = 10^{-3}$, $\Delta T = 5\cdot 10^{-2}$, and $\delta t = 2.5\cdot 10^{-3}$, for both dimensional (Dim) and domain decomposition (DD) splittings.}
		\label{fig:error_wrt_fine}
	\end{figure}
	
	Although the errors in \Cref{fig:error_wrt_fine} seem to decrease exponentially, describing a straight line in logarithmic scale, the graph for the Parareal/FIE-DR method with dimensional splitting changes its trajectory at the 10th iteration. The deterioration of convergence might be caused by the influence of the parameter $s$ on the convergence factor. As shown in \Cref{thm:FIE-FIE} and confirmed in \Cref{conv_FIE-FIE}, the convergence factor of the Parareal/FIE-FIE method is always bound from below by 1/3 and, even if $s$ is changed, its values are similar. On the other hand, if $s$ is increased for the Parareal/FIE-DR method, the interval where $\mathcal{K}_{\text{FIE-DR}}$ is close to 1 is displaced towards the left (see \Cref{conv_FIE-DR}), thus obtaining better convergence for a larger range of eigenvalues.
	
	\Cref{fig:error_wrt_s} confirms the previous claim, plotting the evolution of the error for changing $s$ for domain decomposition splitting. Whereas the convergence seems to worsen for the Parareal/FIE-FIE method for increasing $s$, the Parareal/FIE-DR method behaves just in the opposite manner, even losing the exponential decay of the error. A similar behaviour is observed for dimensional splitting.
	
	\begin{figure}[t]
		\centering
		\subfloat[Parareal/FIE-FIE]{\label{error_wrt_s_FIE}\includegraphics[width=0.465\textwidth]{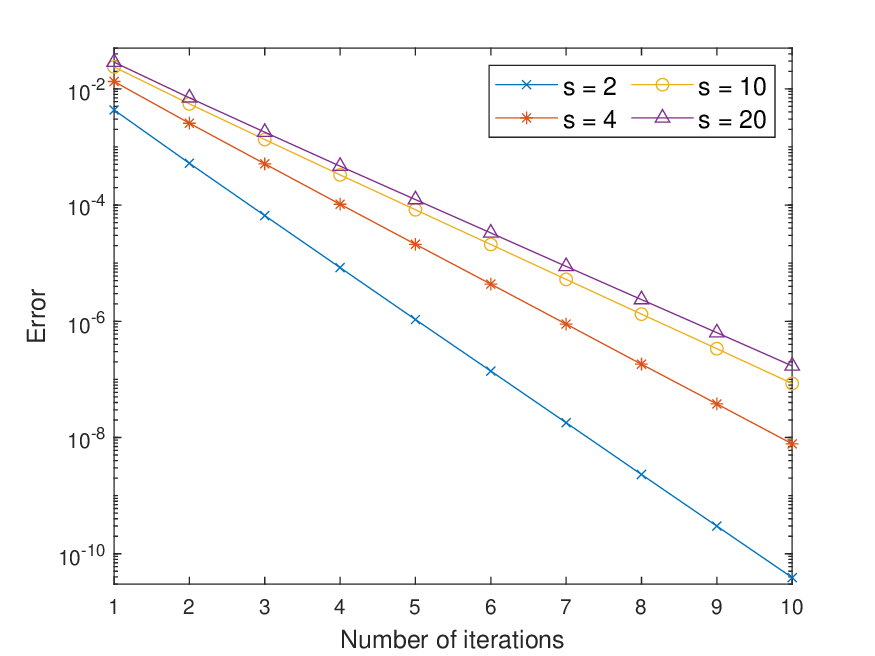}}
		\subfloat[Parareal/FIE-DR]{\label{error_wrt_s_DR}\includegraphics[width=0.465\textwidth]{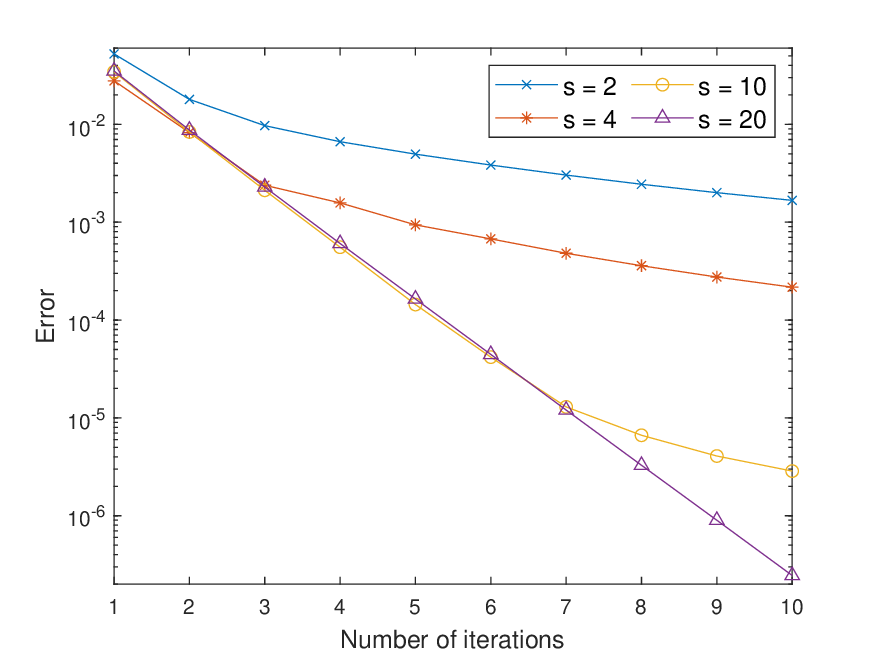}}
		\caption{Evolution of the error of the space-time parallel methods for the domain decomposition splitting with respect to the fine approximation for the test problem \cref{heat_prob} ($d_{11} = 1 = d_{22}$, $d_{12} = 0$) for an increasing number of iterations and changing $s$. We set $T = 1$, $h = 10^{-3}$, and $\Delta T = 5\cdot 10^{-2}$.}
		\label{fig:error_wrt_s}
	\end{figure}
	
	In addition, we analyze the robustness of the proposed solvers with respect to the discretization parameters $\Delta T$ and $h$. Both graphs in \Cref{fig:robustness_1} show the number of iterations required by the methods to achieve tolerance $\varepsilon = 10^{-6}$, for changing $\Delta T$ and $h$. Regarding robustness, the Parareal/IE-IE method and the newly proposed solvers with dimensional splitting show a similar behaviour, with minor variation in the number of iterations. In contrast, for domain decomposition splitting, the methods are robust with respect to $h$, but no longer with respect to the coarse time step $\Delta T$. However, the number of iterations increases at first when the coarse time mesh is refined, whereas for smaller values than $\Delta T = 1.25\cdot 10^{-2}$, the number of iterations is reduced for each new refinement. Thus, the number of iterations required is even reduced for small values of $\Delta T$.
	
	\begin{figure}[t]
		\centering
		\subfloat[Changing $\Delta T$]{\label{robustness_DT}\includegraphics[width=0.465\textwidth]{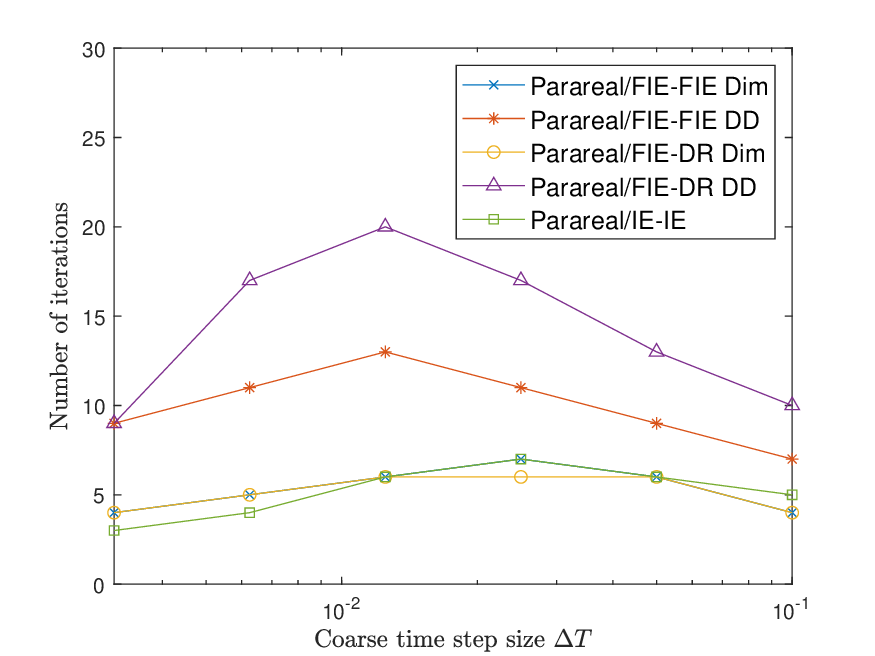}}
		\subfloat[Changing $h$]{\label{robustness_h}\includegraphics[width=0.465\textwidth]{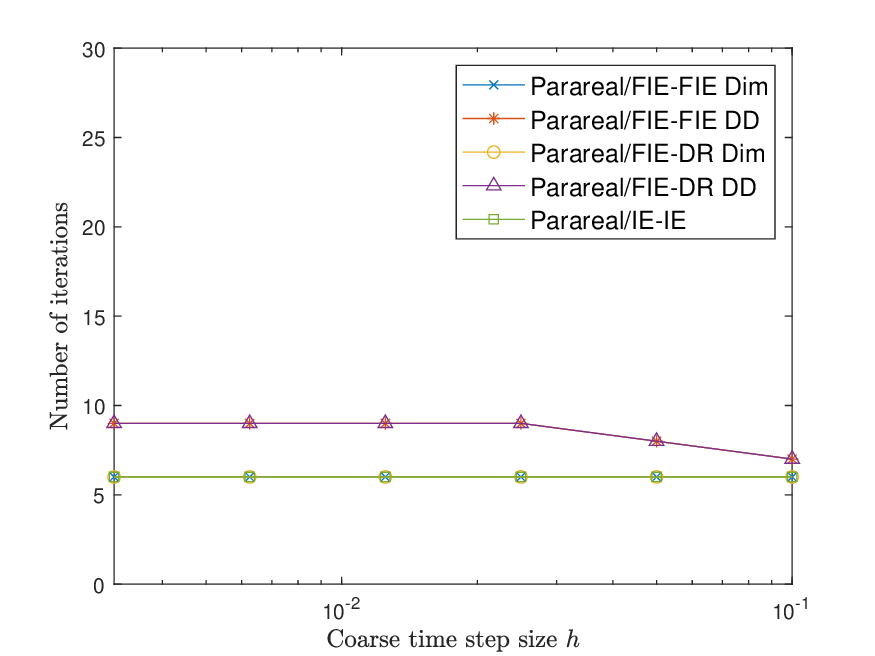}}
		\caption{Number of iterations of the presented methods for the test problem \cref{heat_prob} ($d_{11} = 1 = d_{22}$, $d_{12} = 0$), with $T = 1$, tolerance $\varepsilon = 10^{-6}$, and (a) $h = 0.01$, $s = 20$, and changing $\Delta T$; and (b) $\Delta T =0.05$, $s = 20$, and changing $h$.}
		\label{fig:robustness_1}
	\end{figure}
	
	Furthermore, let us analyze the behaviour of the newly proposed methods when varying the parameters of the domain decomposition, that is, the number of vertical strips for each subdomain, denoted by $q$, and the overlapping size $\beta$. It would be preferable to define a decomposition with as many strips as possible, and, thus, $\beta$ should not be large. \Cref{fig:robustness_2} shows the number of iterations of the methods for changing $q$ and $\beta$, and fixed tolerance $10^{-6}$. Remarkably, the methods are robust with respect to $q$, with really small variation in the number of iterations. However, the parameter $\beta$ influences significantly the behaviour of the methods, especially for the Parareal/FIE-DR method. Overall, the smaller $\beta$, the worse convergence properties of the solvers. Hence, small values of the overlapping size should be avoided. Note that this result is in agreement with the convergence behaviour of classical overlapping Schwarz domain decomposition methods.
	
	\begin{figure}[t]
		\centering
		\subfloat[Changing $q$]{\label{robustness_q}\includegraphics[width=0.465\textwidth]{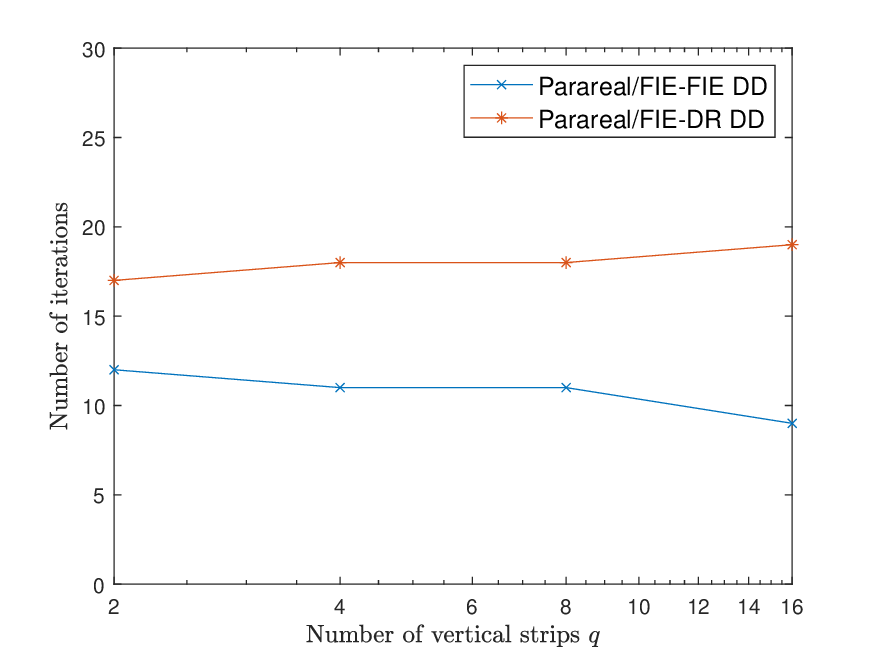}}
		\subfloat[Changing $\beta$]{\label{robustness_beta}\includegraphics[width=0.465\textwidth]{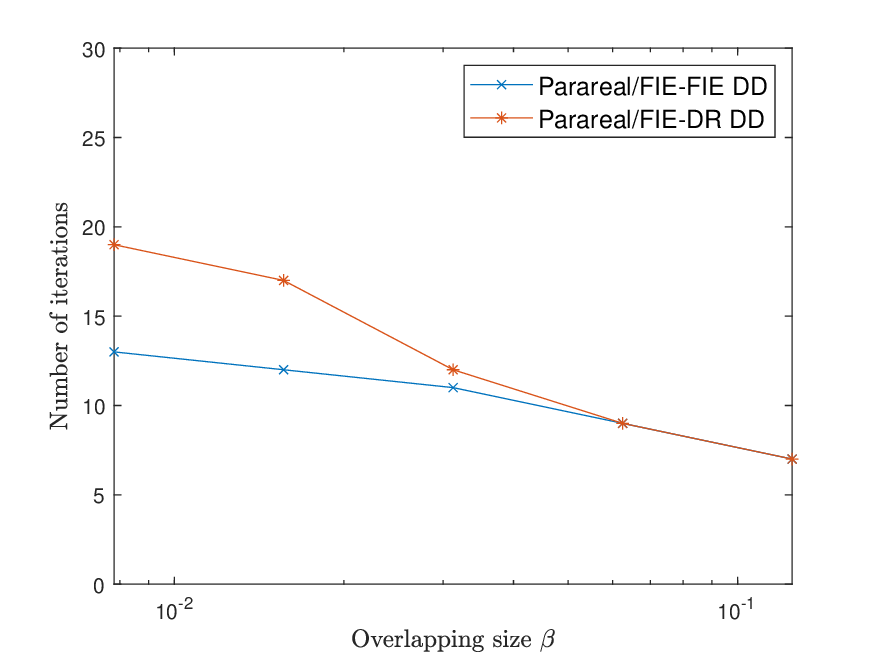}}
		\caption{Number of iterations of the presented methods for the test problem \cref{heat_prob} ($d_{11} = 1 = d_{22}$, $d_{12} = 0$), with $T = 1$, $h = 10^{-3}$, $\Delta = 0.05$, $\delta t = 2.5\cdot 10^{-3}$, and tolerance $\varepsilon = 10^{-6}$, for domain decomposition with (a) changing $q$ and $\beta = 1/64$; and (b) $q = 2$ and changing $\beta$.}
		\label{fig:robustness_2}
	\end{figure}
	
	Finally, a scalability analysis of the proposed methods seems essential in order to study the strengths and weaknesses of their parallelization power. This experiment considers only domain decomposition splittings for the partition of the elliptic operator, in the light of the versatility of the technique. We also restrict our attention to the Parareal/FIE-FIE method, since the Parareal/FIE-DR method performs similarly regarding computational time.
	
	Then, we set $h = 0.01$, $\Delta T = 1/32$, and $\delta t = 1/3200$ as the numerical parameters of the discrete problem. Moreover, let us consider two different domain decompositions for the two propagators. The number of strips will change for the coarse propagator, considering a constant overlapping size $\beta = 1/32$, whereas the decomposition will consist of two vertical strips and overlapping size $\beta = 1/16$ for the fine propagator. 
	
	In particular, we consider $q = 2$, $q = 4$, and $q = 8$ as the vertical strips for the domain decomposition which corresponds to the coarse propagator. Thus, the errors of the three different coarse propagators are $2.1373\cdot 10^{-1}$, $3.2813\cdot 10^{-1}$, and $2.2737\cdot 10^{-1}$, respectively. On the other hand, the error of the fine approximation is $5.2640\cdot 10^{-3}$ in this setting. Following the suggestions in \cite{ArteagaRuprechtKrause2015}, we set the number of iterations of the algorithm to 5, since after five iterations the obtained errors are $3.3274\cdot 10^{-3}$, $5.5688\cdot 10^{-3}$, and $5.9709\cdot 10^{-3}$, that is, about the same order of magnitude of the error obtained from the fine approximation.
	
	In this setting, the speedup is defined as
	\begin{displaymath}
		S(q) = \frac{T_{s}}{T_{p}(q)},
	\end{displaymath}
	where $T_s$ denotes the runtime in seconds when the problem is solved sequentially, and $T_{p}(q)$ stands for the runtime in seconds for solving the problem in parallel using $q$ processors. The numerical calculations are performed on a shared-memory machine with 256 GB of RAM that has 192 AMD CPU Epyc7513 2.6 GHz processors, with CESAR (Supercomputation Center of Aragon) being in charge of maintenance. 
	
	\Cref{fig:scalability} shows the runtime (left) and speedup (right) of the Parareal/FIE-FIE method in the aforementioned framework. Remarkably, promising scalability results are obtained that behave similarly to those corresponding to the parareal algorithm with no partitioning (cf. \cite{ArteagaRuprechtKrause2015}), with an improvement for larger number of processors. In addition, a certain improvement of the speedup is observed for an increasing number of vertical strips. It is noteworthy that, for an increasing number of processors, the speedup differences with respect to $q$ can be relevant, since the computational cost of the fine propagator decreases when more processors are available.
	
	\begin{figure}[t]
		\centering
		\subfloat[Runtime in seconds]{\label{CPUtime}\includegraphics[width=0.465\textwidth]{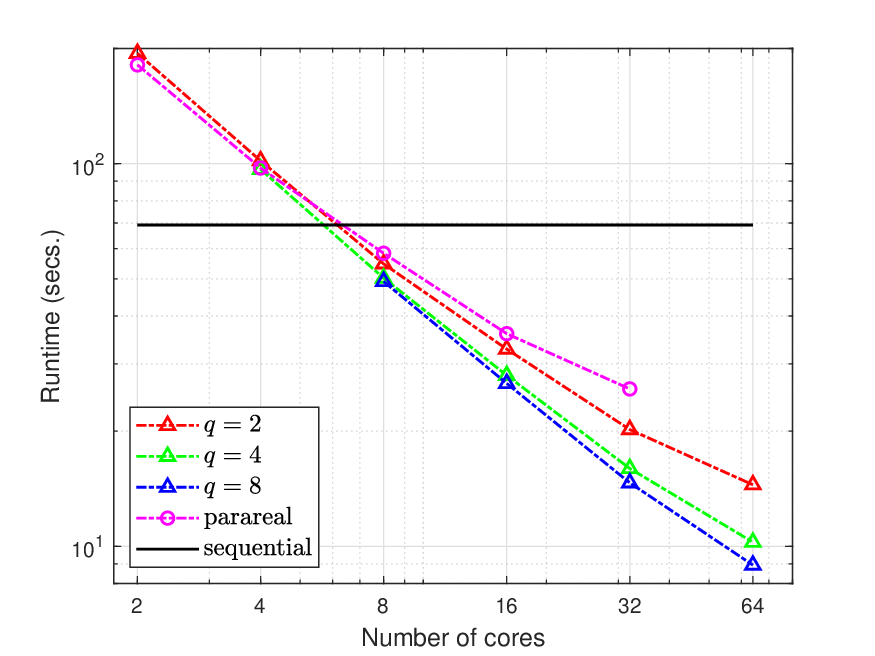}}
		\subfloat[Speedup]{\label{speedup}\includegraphics[width=0.465\textwidth]{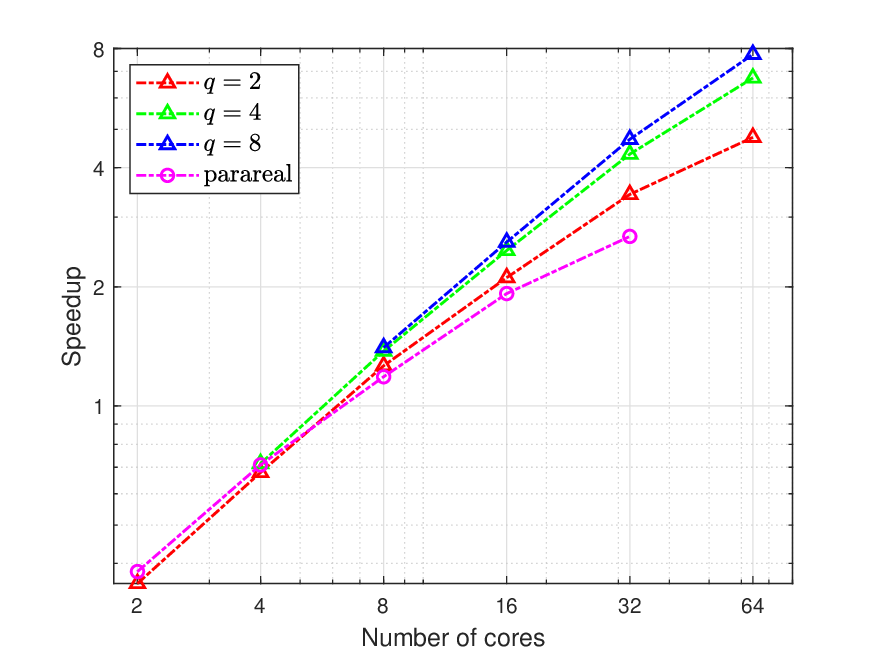}}
		\caption{Scalability analysis for the Parareal/FIE-FIE methos with domain decomposition splitting applied to the test problem \cref{heat_prob} ($d_{11} = 1 = d_{22}$, $d_{12} = 0$) for 5 iterations, with $h = 0.01$, $\Delta T = 1/32$, $\delta t = 1/3200$. We consider $q = 2,4,8$ and $\beta = 1/32$ for the coarse propagator and $q = 2$ and $\beta = 1/16$ for the fine propagator, and we also consider the parareal algorithm without partitioning.}
		\label{fig:scalability}
	\end{figure} 
	
	The results in \cite{ArrarasPorteroYotov2014} offer further insight into the scalability analysis of domain decomposition splitting on its own. Regarding dimensional splitting, although communication among processors increases, it seems to outperform the scalability results obtained for domain decomposition. The main reason is the simpler structure of the resulting matrices, since tridiagonal matrices (instead of pentadiagonal ones) are obtained in this case. Furthermore, fewer iterations are required in order to reach the desired tolerance.
	
	\subsection{Experiments with a full diffusion tensor}
	\label{susec:Dfull}
	
	Let us increase the complexity of problem \cref{heat_prob} by considering a diffusion tensor whose components are
	\begin{displaymath}
		d_{11}(\mathbf{x}) = d_{11}(x,y)  = \frac{1}{2 + \cos(3\pi x)\cos(2\pi y)} = d_{22}(x,y) = d_{22}(\mathbf{x}),
	\end{displaymath} 
	and $d_{12}(\mathbf{x}) = 1/4$, for $\mathbf{x} = (x,y) \in \Omega$. This tensor was considered as well in \cite{ArrarasPortero2015}. As explained in that reference, since the elliptic operator contains mixed partial derivatives, it is partitioned using the more flexible domain decomposition splitting technique. 
	
	We henceforth set $h = 10^{-3}$, $\Delta T = 5\cdot 10^{-2}$, and $\delta t = 2.5\cdot 10^{-3}$. Moreover, the same setting of the domain decomposition splitting as for the previous subsection is considered. Then, \Cref{fig:error_wrt_fine_Dfull} shows the evolution of the error of the approximate solutions with respect to the fine approximation for increasing number of iterations. The space-time parallel solvers converge more slowly than the Parareal/IE-IE method. Remarkably, since mixed partial derivatives are considered, systems of nonadiagonal matrices must be solved. Therefore, partitioning the elliptic operator by domain decomposition splitting yields faster performance of the solvers. Actually, for a sufficiently large number of processors, it can be faster than the Parareal/IE-IE method to reach certain tolerance.
	
	\begin{figure}[t]
		\centering
		\includegraphics[width=0.5\textwidth]{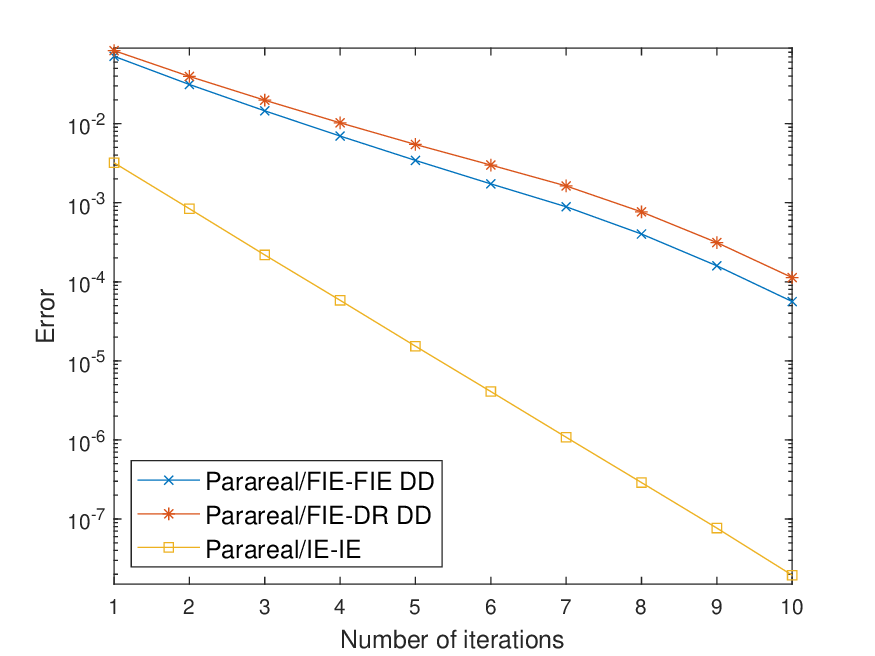}
		\caption{Evolution of the error of the space-time parallel methods and the Parareal/IE-IE method with respect to the fine approximation for the test problem \cref{heat_prob} (full $D$) for an increasing number of iterations. We set $T = 1$, $h = 10^{-3}$, $\Delta T = 5\cdot 10^{-2}$, and $\delta t = 2.5\cdot 10^{-3}$.}
		\label{fig:error_wrt_fine_Dfull}
	\end{figure}
	
	Once again, we analyze the influence of the parameter $s$ on the convergence of the methods. \Cref{fig:error_wrt_s_Dfull} depicts the evolution of the errors with respect to the fine approximations for the solutions of problem \cref{heat_prob}, with full tensor $D$, computed by both space-time parallel solvers and changing $s$. The Parareal/FIE-FIE method behaves similarly for the two problems under consideration throughout this section, that is, the lower the qoutient $s = \Delta T/\delta t$, the better convergence properties of the solver. Unlike for the problem with diagonal diffusion tensor with constant coefficients, for problem \cref{heat_prob} with a space-dependent full tensor, the influence of $s$ on the convergence of the Parareal/FIE-DR method is not that clear. Although for $s = 2$ there exists a significant deterioration of the convergence, for $s \geq 4$ no major difference can be found. The dependence of the diffusion coefficients on space seems to be the main reason for this behaviour. All in all, the increment of $s$ does not imply any deterioration in terms of convergence for the Parareal/FIE-DR method, even for such complex problems.
	
	\begin{figure}[t]
		\centering
		\subfloat[Parareal/FIE-FIE]{\label{error_wrt_s_FIE_Dfull}\includegraphics[width=0.465\textwidth]{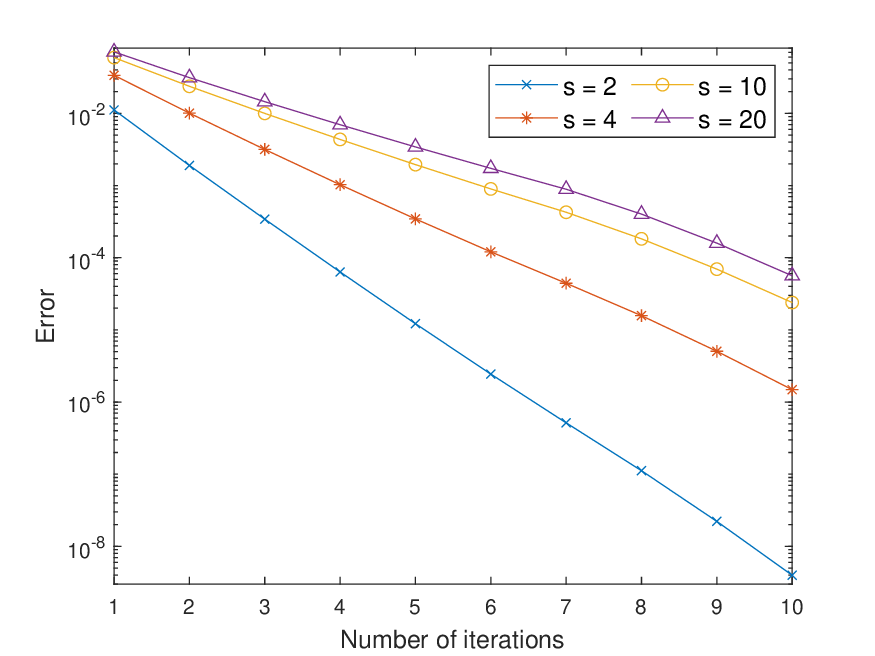}}
		\subfloat[Parareal/FIE-DR]{\label{error_wrt_s_DR_Dfull}\includegraphics[width=0.465\textwidth]{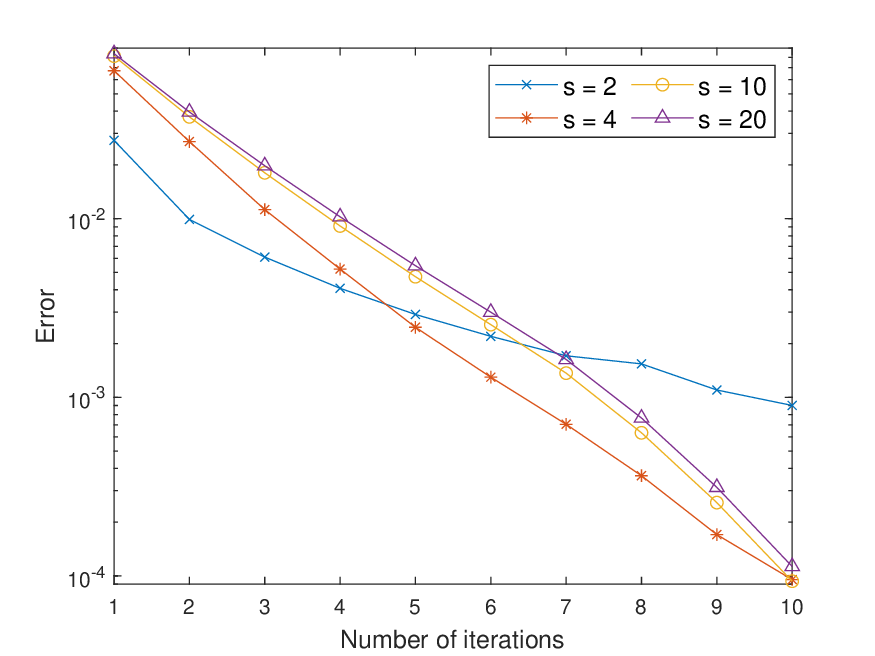}}
		\caption{Evolution of the error of the space-time parallel methods for the domain decomposition splitting with respect to the fine approximation for the test problem \cref{heat_prob} (full $D$) for an increasing number of iterations and changing $s$. We set $T = 1$, $h = 10^{-3}$, and $\Delta T = 5\cdot 10^{-2}$.}
		\label{fig:error_wrt_s_Dfull}
	\end{figure}
	
	\section{Concluding remarks}
	\label{sec:concl}
	
	The parareal algorithm has turned out to be a judicious choice in order to design space-time parallel solvers. Our approach is based on the combination of such an algorithm with splitting techniques. As shown in the present paper, suitable choices of splitting time integrators as the coarse and fine propagators ensure promising stability and convergence properties of the resulting solvers for the integration of parabolic problems. 
	
	As a matter of fact, one of the drawbacks of the parareal algorithm is solved by the newly proposed methods, that is, the large number of calculation cores which remain unused while the coarse propagator is performing. By considering time-splitting integrators for this propagator, the processors are invested in the parallelization in space, yielding faster computations of the propagator. This strategy has shown to improve the speedup of the proposed method, and, overall, satisfactory scalability properties have been obtained.
	
	Two partitioning strategies have been proposed in this work, i.e., dimensional and domain decomposition splittings. Although both partitioning criteria for the elliptic operator guarantee convergence for diffusion-dominated problems, as proven in \Cref{sec:conv}, dimensional splitting would appear to converge faster (even faster than for the unsplit version of the method). However, domain decomposition splitting has been shown to be more flexible, being able to consider, for instance, elliptic terms containing mixed partial derivatives (see \Cref{susec:Dfull}). In addition, both splitting techniques could be extended to 3D spatial domains, and, in particular, domain decomposition splitting could be adapted to more complex spatial discretizations, preserving the stability and convergence properties proven in this paper.
	
	Regarding the influence of the quotient between the coarse and fine time steps, namely, $s =\Delta T/\delta t$, the behaviour of both Parareal/FIE-FIE and Parareal/FIE-DR methods is the opposite. While the increment of $s$ implies certain deterioration of convergence in the former case, it appears that the convergence properties of the solver improve in the latter case when $s$ is increased.
	
	On the whole, the space-time parallel iterative solvers designed in this paper require few iterations to compute accurate approximations, with optimal use of the available processors. As a topic of future research, higher-order splitting time integrators should also be considered, specially for the fine propagator, in order to obtain more accurate fine approximations without loss of parallelization power. All in all, the key idea of our proposal seems to be an easy and powerful strategy for the construction of space-time parallel methods, replacing the parareal algorithm by some other relevant parallel-in-time integrators, such as, for example, the well-known MGRIT method (cf. \cite{FalgoutFriedhoffKolevMaclachlanSchroder2014}).
	
	\bibliographystyle{siamplain}
	\bibliography{references}
\end{document}